\documentclass[12pt,a4paper,reqno,tbtags]{amsart}

\usepackage[margin=1.1in]{geometry}
\usepackage{amsmath,amssymb,amsthm,mathtools}
\usepackage{microtype}
\usepackage{graphicx}
\usepackage{caption}
\usepackage{subcaption}
\usepackage[hidelinks]{hyperref}
\hypersetup{
 pdftitle={The Y-partition is the optimal three-partition for the disc and the harmonic oscillator},
 pdfauthor={Mikael Sundqvist},
 pdfsubject={Spectral minimal partitions and radial transplantation},
 pdfkeywords={spectral partitions, disc, harmonic oscillator, radial transplantation}}

\numberwithin{equation}{section}
\newtheorem{theorem}{Theorem}[section]
\newtheorem{proposition}[theorem]{Proposition}
\newtheorem{lemma}[theorem]{Lemma}
\newtheorem{corollary}[theorem]{Corollary}
\theoremstyle{definition}

\theoremstyle{remark}
\newtheorem{remark}[theorem]{Remark}

\newcommand{\R}{\mathbb{R}}
\newcommand{\Disk}{B}
\newcommand{\Sphere}{\mathbb{S}^{2}}
\newcommand{\Circle}{\mathbb{S}^{1}}
\newcommand{\dd}{\mathop{}\!{d}}
\newcommand{\QB}{Q_{\Disk}}
\newcommand{\QS}{Q_{\Sphere}}
\newcommand{\RB}{\mathcal{R}_{\Disk}}
\newcommand{\RS}{\mathcal{R}_{\Sphere}}
\newcommand{\partenergy}{\mathfrak{L}}
\newcommand{\sector}{\Sigma}
\newcommand{\QH}{\mathcal Q_H}
\newcommand{\QHo}{Q_H}
\newcommand{\RH}{\mathcal R_H}
\newcommand{\DefH}{\mathcal D_H}
\newcommand{\TransH}{T_H}
\newcommand{\LH}{\partenergy_3(H)}
\DeclareMathOperator{\supp}{supp}
\DeclareMathOperator{\Trace}{Tr}

\DeclarePairedDelimiter{\measure}{\lvert}{\rvert}
\DeclarePairedDelimiter{\abs}{\lvert}{\rvert}
\DeclarePairedDelimiter{\norm}{\lVert}{\rVert}

\DeclarePairedDelimiter{\set}{\lbrace}{\rbrace}

\title
  [Spectral three-partitions]
  {The Y-partition is the optimal\\
   three-partition for the disc\\
   and the harmonic oscillator}

\author{Mikael Sundqvist}

\address
  [M. Sundqvist]
  {Department of Mathematics, Lund University, Sweden}

\email{mikael.persson\_sundqvist@math.lth.se}



\begin{document}

\begin{abstract}
We prove that the Y-partition into three equal sectors is the minimal
spectral three-partition both for the Dirichlet Laplacian on the unit disc
and for the planar harmonic oscillator $-\Delta+\abs{x}^2$, with minimal
energies $j_{3/2,1}^{2}$ and $5$; for the disc, this confirms a conjecture
of Helffer and Hoffmann-Ostenhof. The minimizing regular strong partition is
unique up to rotation, and every open minimizing partition has cells with
the Dirichlet form domains of the sectors. The proof is a positive radial
transplantation to the sphere that preserves segregation and matches the
angular-energy measures of the separated model states; the three-lune
theorem of Helffer, Hoffmann-Ostenhof, and Terracini then gives the lower
bound. The transplantation lowers the shifted quadratic form by a
nonnegative defect with strictly positive radial weight; in the equality
case, spherical equipartition makes the defects vanish, which separates
variables, and a Poincar\'e inequality on the circle identifies the sectors.
\end{abstract}

\maketitle

\section{Introduction}

The equal-sector ground states for the Dirichlet Laplacian on the disc
and for the harmonic oscillator on the plane have different radial
factors but the same angular structure. We compare both
problems with the three-lune partition of the sphere. The comparison
matches angular-energy measures rather than area measures, and its
strict radial defect gives uniqueness.
We develop the disc problem first and then apply the same construction
to the oscillator in Section~\ref{ho:sec-oscillator}.

\subsection{The disc problem and the main statements}\label{sec:introduction}
Let $\Disk=\set{x\in\R^2\colon\norm{x}<1}$ be the unit disc. Throughout this
article, all functions are real-valued. For a nonempty open set $D\subset\Disk$,
we define its first Dirichlet eigenvalue by
\[
 \lambda_1(D)=
 \inf_{0\ne u\in H^1_0(D)}
 \frac{\int_D\norm{\nabla u}^2\dd x}
      {\int_Du^2\dd x}.
\]
Here $H^1_0(D)$ is the closure of $C_c^\infty(D)$ in $H^1(D)$.
We always extend functions in this space by zero to $\Disk$, and regard
$H^1_0(D)$ as the corresponding closed subspace
of $H^1_0(\Disk)$. In particular, equalities between Dirichlet form
domains below refer to these zero-extended subspaces.

An open three-partition is a triple $(D_1,D_2,D_3)$ of pairwise disjoint,
nonempty, connected open subsets of $\Disk$. We do not require the cells to
cover the disc. The energy of such a partition and the corresponding minimal
energy are
\begin{equation}\label{eq:partition-energy}
 \Lambda(D_1,D_2,D_3)=\max_{1\le i\le3}\lambda_1(D_i),
 \qquad
 \partenergy_3(\Disk)=\inf_{(D_1,D_2,D_3)}\Lambda(D_1,D_2,D_3).
\end{equation}
We impose no regularity on the competitors in~\eqref{eq:partition-energy}.
A partition is \emph{strong} if the relative closures of its cells cover the
ambient domain. A \emph{regular strong representative} has an interface
consisting of finitely many smooth arcs meeting at finitely many vertices; its
cells are the connected components of the complement, and isolated punctures
are excluded. The infimum over regular strong representatives has the same
value; see~\cite[Theorems~1.12 and~4.14]{HHTnodal}.

Write $j_{\nu,n}$ for the $n$th positive zero of the Bessel function $J_\nu$.
Throughout, we fix the constants
\[
 \nu=\frac32,\qquad \kappa=j_{3/2,1},\qquad
 L=\kappa^2,\qquad \mu=\nu(\nu+1)=\frac{15}{4}.
\]
Numerically, $\kappa\approx4.4934$ and $L\approx20.191$.
The three sectors
\begin{equation}\label{eq:Y-sectors}
 \sector_i(\alpha)=
 \set[\Big]{(r,\phi)\colon 0<r<1,\quad
 \alpha+\frac{2\pi(i-1)}3<\phi<\alpha+\frac{2\pi i}3},
 \qquad i=1,2,3,
\end{equation}
with angles interpreted modulo $2\pi$, form the \emph{Y-partition} of the
disc (Figure~\ref{fig:partitions}(a)). We use this name also for its
rotated copies, for the partition of the plane into three sectors of
opening $2\pi/3$ with vertex at the origin, and for the partition of the
sphere into three equal lunes (Figure~\ref{fig:partitions}(b)), up to
rotation. The planar configuration is also called the \emph{Mercedes
partition} or \emph{Mercedes star}; see~\cite[Section~1]{HHTsphere}
and~\cite{DiskChapter}.

\begin{figure}[tbp]
  \centering
  \subcaptionbox
    {}
    [0.4\textwidth]
    {\includegraphics[height=5cm,page=1]{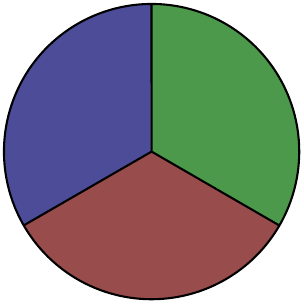}}
  \hskip1cm
  \subcaptionbox
    {}
    [0.4\textwidth]
    {\includegraphics[height=5cm,page=2]{graphics.pdf}}
    \caption{The Y-partitions of (a) the disc and (b) the sphere.}
    \label{fig:partitions}
\end{figure}

\begin{theorem}\label{thm:main}
The minimal three-partition energy of the unit disc is
\[
 \partenergy_3(\Disk)=j_{3/2,1}^{2}.
\]
Every regular strong representative of a minimizing partition is a
Y-partition, up to rotation and relabeling. More generally, for every minimizing
open partition $(D_1,D_2,D_3)$, there exist $\alpha\in\R$ and a relabeling such that
\begin{equation}\label{eq:cell-form-domains}
 D_i\subset\sector_i(\alpha),\qquad
 H^1_0(D_i)=H^1_0\bigl(\sector_i(\alpha)\bigr),
 \qquad i=1,2,3,
\end{equation}
where the spaces are identified by zero extension to $\Disk$. There are
constants $M_i>0$ such that the zero-extended nonnegative ground states of the
cells are given almost everywhere on $\Disk$ by
\begin{equation}\label{eq:sector-groundstates}
 u_i(r,\phi)=
 \begin{cases}
 M_i J_{3/2}(\kappa r)
 \sin\Bigl(\dfrac32(\phi-\alpha_i)\Bigr),
   & (r,\phi)\in\sector_i(\alpha),\\
 0,&\text{otherwise},
 \end{cases}
 \qquad \alpha_i=\alpha+\frac{2\pi(i-1)}3.
\end{equation}
\end{theorem}

\begin{remark}\label{rem:form-domains}
Regularity enters Theorem~\ref{thm:main} only in the identification of the
cells as sets. For arbitrary open cells, uniqueness is stated as equality of
the zero-extended spaces $H^1_0(D_i)$ and $H^1_0(\sector_i(\alpha))$. This
preserves all Dirichlet spectral information while allowing modifications
that do not affect the Dirichlet form, such as deleting an interior point
from a cell. Section~\ref{sec:partitions} proves the form-domain
identification and shows that regular strong representatives admit no such
deletions.
\end{remark}

We also use the Rayleigh quotient on the whole disc. For
$0\ne u\in H^1_0(\Disk)$, set
\[
 \RB(u)=\frac{\int_{\Disk}\norm{\nabla u}^2\dd x}{\int_{\Disk}u^2\dd x}.
\]
Three nonnegative nonzero functions are \emph{segregated} if $u_iu_j=0$
almost everywhere whenever $i\ne j$. The zero-extended ground states of a
three-partition therefore form a segregated triple with
$\RB(u_i)=\lambda_1(D_i)$. The optimality assertion of
Theorem~\ref{thm:main} is proved directly in Section~\ref{sec:optimality}.
For uniqueness, we establish the following stronger statement about
segregated Sobolev functions. Its equality case identifies the ground states
of every minimizing open partition, without any prior choice of interfaces.

\begin{theorem}\label{thm:functional}
For every segregated triple $(u_1,u_2,u_3)$ in $H^1_0(\Disk)$,
\begin{equation}\label{eq:functional-bound}
 \max_i\RB(u_i)\ge L.
\end{equation}
Equality holds if and only if, for some $\alpha\in\R$, some constants
$M_i>0$, and some relabeling, the $u_i$ are the sector ground states
\eqref{eq:sector-groundstates}, where all equalities of functions are
understood almost everywhere.
\end{theorem}

\subsection{The planar harmonic oscillator}\label{ho:intro}

The same method applies to the planar harmonic oscillator
\begin{equation}\label{ho:operator}
 H=-\Delta+\abs{x}^2\qquad\text{in }\R^2.
\end{equation}
Its quadratic form and form domain are
\begin{equation}\label{ho:form}
 q_H[u]=\int_{\R^2}\bigl(\abs{\nabla u}^2+\abs{x}^2\abs{u}^2\bigr)\dd x,
 \qquad
 \QH=\{u\in H^1(\R^2): \abs{x}u\in L^2(\R^2)\}.
\end{equation}
We use the norm $\norm{u}_{\QH}^2=q_H[u]+\norm{u}_2^2$. For an open set
$D\subset\R^2$, define
\[
 \QH(D)=\overline{C_c^\infty(D)}^{\QH},
 \qquad
 \lambda_H(D)=\inf_{0\ne u\in\QH(D)}\frac{q_H[u]}{\norm{u}_2^2}.
\]
Functions on $D$ are extended by zero to the plane.

As for the disc, an open three-partition is a triple
$\mathcal P=(D_1,D_2,D_3)$ of pairwise disjoint, nonempty, connected open
sets. The cells need not exhaust the plane. Set
\begin{equation}\label{ho:partition-energy}
 \Lambda_H(\mathcal P)=\max_{1\le i\le3}\lambda_H(D_i),
 \qquad
 \LH=\inf_{\mathcal P}\Lambda_H(\mathcal P).
\end{equation}
Regular strong representatives are understood as in
Section~\ref{sec:introduction}, with unbounded interface arcs allowed.
Write
\begin{equation}\label{ho:equal-sectors}
 \sector_i^\infty(\alpha)=
 \left\{(r,\phi): r>0,\quad
 \alpha+\frac{2\pi(i-1)}3<\phi<\alpha+\frac{2\pi i}3\right\},
 \qquad i=1,2,3,
\end{equation}
with angles interpreted modulo $2\pi$. These sectors form the Y-partition
of the plane, with vertex at the origin.

\begin{theorem}\label{ho:main-partition}
The minimal three-partition energy of the planar harmonic oscillator is
\begin{equation}\label{ho:main-value}
 \LH=5.
\end{equation}
Every regular strong minimizing representative is, up to rotation and
relabeling, the Y-partition~\eqref{ho:equal-sectors}.
More generally, for every minimizing open partition, there are an
$\alpha\in\R$ and a relabeling such that
\begin{equation}\label{ho:cell-form-domains}
 D_i\subset\sector_i^\infty(\alpha),\qquad
 \QH(D_i)=\QH\bigl(\sector_i^\infty(\alpha)\bigr),
 \qquad i=1,2,3,
\end{equation}
where the spaces are identified by zero extension to $\R^2$.
\end{theorem}

The ground states of the minimizing sectors have radial factor
$h(r)=r^{3/2}e^{-r^2/2}$ and the same angular factors as
in~\eqref{eq:sector-groundstates}. Section~\ref{ho:sec-oscillator} proves
Theorem~\ref{ho:main-partition} and its functional counterpart,
Theorem~\ref{ho:main-functional}, and gives the scaling, translation, and
Gaussian formulations.

\subsection{The idea of the proof}

We start from the spherical three-partition theorem of Helffer,
Hoffmann-Ostenhof, and Terracini: three equal lunes minimize the energy on the
sphere, with value $\mu=15/4$; see~\cite[Theorem~1.1]{HHTsphere}, restated as
Theorem~\ref{thm:spherical-input} below (see also Figure~\ref{fig:partitions}).
The sector and lune ground states share their angular factor. We seek a positive
linear radial transplantation that preserves this factor and replaces the disc
radial state by the spherical one. Such a map preserves segregation. The desired
comparison is between the quadratic forms shifted by $L$ and $\mu$: a function
with Rayleigh quotient below $L$ on the disc should be mapped to a function with
quotient below $\mu$ on the sphere. Section~\ref{sec:transplantation} constructs
this comparison, and Section~\ref{sec:optimality} applies the spherical theorem
to exclude a three-partition of the disc with energy below $L$.

For the disc uniqueness statement, we analyze equality in the comparison.
Matching the angular energies, we obtain a nonnegative defect that measures
radial variation after factoring out the disc radial state. The strict stretching
inequality of Proposition~\ref{prop:stretching} shows that its radial weight is
positive, and Proposition~\ref{prop:full-comparison} extends the comparison, with
a localized form of the defect estimate, to the full Sobolev spaces.
Section~\ref{sec:sphere-input} proves a saturation statement: if the Rayleigh
quotients of three segregated functions on the sphere are all at most $\mu$, then
they all equal $\mu$. This follows from the equipartition theorem for minimal
spectral partitions, and it shows that all three defects vanish in the equality
case. Vanishing defects force separation of variables, after which the interval
Poincar\'e inequality determines the three angular sectors.
Section~\ref{sec:proofs} completes the rigidity argument and passes from ground
states to cells.

For the oscillator, we replace the Bessel factor $R$ by $h$ and match
its angular-energy measure with the same spherical measure. A comparison
of derivatives of the normalized cumulative densities proves the strict
radial inequality. The resulting transplantation satisfies an exact
defect identity on the oscillator form domain. Spherical saturation makes
all three defects vanish in the equality case, and the same angular
rigidity lemma identifies the Y-partition. Thus the two uniqueness proofs
share their spherical and angular inputs; neither uses spherical
uniqueness.

\subsection{Background}\label{sec:historical-background}
The functional in~\eqref{eq:partition-energy} is the maximum of the cell
eigenvalues. Among the functionals given by the $p$-norms of the vector of
cell eigenvalues, it is the $p=\infty$ member. Minimizing the sum
or a finite-$p$ mean of the eigenvalues gives a different variational
problem; the relation between these functionals is explained in Helffer's
survey and in the comparison by Helffer and
Hoffmann-Ostenhof~\cite{HelfferSurvey,HHnotions}.

The general existence theory for optimal partitions was developed by
Bucur, Buttazzo, and Henrot~\cite{BucurButtazzoHenrot}. Conti, Terracini,
and Verzini related eigenvalue partitions to competing and segregating
states~\cite{CTV03,CTV05fucik,CTV05segregation}. Caffarelli and Lin developed
the associated regularity theory~\cite{CaffarelliLin,CaffarelliLin2008}.
Tavares and Terracini established regularity of the nodal set of segregated
critical configurations under a weak reflection
law~\cite{TavaresTerracini}. A complementary line of work asks when a
spectral partition is generated by the nodal domains of an eigenfunction;
see the work of Ancona, Helffer, and Hoffmann-Ostenhof~\cite{Ancona2004}.

Helffer, Hoffmann-Ostenhof, and Terracini (HHT) established the connection
with Courant-sharp eigenfunctions and the existence of regular strong
representatives~\cite{HHTnodal}. A partition is \emph{bipartite} when its
adjacency graph can be colored with two colors; an eigenvalue $\lambda_k$
is \emph{Courant-sharp} when an associated eigenfunction has exactly $k$
nodal domains. HHT proved that a bipartite minimal partition is nodal, and
that $\partenergy_k=\lambda_k$ exactly when $\lambda_k$ is Courant-sharp.

For the disc, write $\lambda_j=\lambda_j(\Disk)$, counting multiplicity.
The first eigenvalues are
\[
 \lambda_1=j_{0,1}^2,\qquad
 \lambda_2=\lambda_3=j_{1,1}^2,\qquad
 \lambda_4=\lambda_5=j_{2,1}^2.
\]
The eigenspaces corresponding to $\lambda_2=\lambda_3$ and
$\lambda_4=\lambda_5$ produce two half-discs and four quarter-discs,
respectively. Nonradial separated disc eigenfunctions have
an even number of nodal domains, so the first level admitting three nodal
domains is $\lambda_{15}=j_{0,3}^2$, whose eigenfunction is radial with
two nodal circles. The Bessel-zero ordering and the nodal-partition
theorems give
\[
 \lambda_1<\partenergy_2(\Disk)=\lambda_2=\lambda_3
 <\partenergy_3(\Disk)<\lambda_4=\partenergy_4(\Disk).
\]
The ordering of the Bessel zeros and the nodal-partition identities used here are
discussed in~\cite[Section~9 and Appendix~A]{HHTnodal} and
\cite[Section~1.2]{BHdisk}; see also~\cite{HelfferPerssonSundqvist2016} for the
Courant-sharp classification in Euclidean balls. Thus the minimal three-partition
of the disc is not a nodal partition.

The Y-partition conjecture for the disc goes back to unpublished notes by Helffer
and Hoffmann-Ostenhof written in December 2005, cited in~\cite{HHTnodal}. The CRM
chapter~\cite{DiskChapter} formulates the conjecture and develops a double-cover
approach. Under a simple-connectivity hypothesis in the punctured disc, a minimal
three-partition lifts to a symmetric six-partition of the double cover, and a
Courant-sharp eigenvalue on that cover identifies its projection as the
Y-partition.

For a regular strong representative of a minimal three-partition of the disc, the
Euler-formula classification recalled by Bonnaillie-No\"el and Helffer
\cite[Proposition~4.1 and Figure~2]{BHdisk} gives the six boundary-incidence
configurations shown in Figure~\ref{fig:topologicalconfigurations}. Their
one-junction theorem \cite[Proposition~1.4]{BHdisk} shows that a minimizing
partition with exactly one interior junction is the Y-partition, without assuming
the junction's location or any symmetry. It therefore identifies the minimizer in
type~(a) and excludes types~(b) and~(c), but does not eliminate the two-junction
alternatives~(d)--(f). Helffer's survey~\cite{Helffer2011} discusses these
questions together with the rectangle and sphere problems.

\begin{figure}[tbp]
  \centering
  \subcaptionbox
    {}
    [0.3\textwidth]
    {\includegraphics[height=3.5cm,page=3]{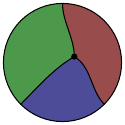}}
  \hfill
  \subcaptionbox
    {}
    [0.3\textwidth]
    {\includegraphics[height=3.5cm,page=4]{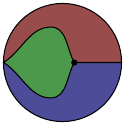}}
  \hfill
  \subcaptionbox
    {}
    [0.3\textwidth]
    {\includegraphics[height=3.5cm,page=5]{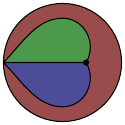}}
  \\[.5cm]
  \subcaptionbox
    {}
    [0.3\textwidth]
    {\includegraphics[height=3.5cm,page=6]{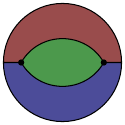}}
  \hfill
  \subcaptionbox
    {}
    [0.3\textwidth]
    {\includegraphics[height=3.5cm,page=7]{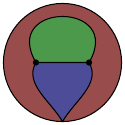}}
  \hfill
  \subcaptionbox
    {}
    [0.3\textwidth]
    {\includegraphics[height=3.5cm,page=8]{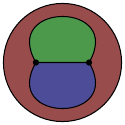}}
    \caption{Schematic interface configurations in the Euler-formula
    classification~\cite[Proposition~4.1]{BHdisk}. Dots mark interior triple
    junctions. In (a)--(c), the three arcs end on the boundary circle at
    three, two, or one distinct points, respectively. In (d) and (e), the two
    arcs end there at two and one distinct points, respectively. In (f), the
    interface is an entirely interior theta-network.
    The drawings represent connectivity and boundary incidence, not shapes
    optimized for the spectral problem.}
    \label{fig:topologicalconfigurations}
\end{figure}

Numerical investigations support the equal-sector
candidate~\cite{BHV,BogoselBonnaillieNoel2018}.
The Aharonov--Bohm approach was introduced by Bonnaillie-No\"el, Helffer,
and Hoffmann-Ostenhof~\cite{BHHO2009}. In the magnetic formulations of
Noris and Terracini and of Helffer and Hoffmann-Ostenhof, a minimal partition
is the nodal partition of an Aharonov--Bohm eigenfunction~\cite{NorisTerracini,Magnetic}, and the number
and locations of the poles become part of the geometric problem. For a
single moving pole, Bonnaillie-No\"el, Noris, Nys, and Terracini studied
the dependence of the eigenvalues on its position~\cite{BNNNT2014}. The
survey~\cite{BHsurvey} discusses these approaches.

The spherical problem provides the model for our comparison. Bishop
formulated the conjecture that the Y-partition minimizes the mean of the
first eigenvalues of the three cells~\cite{Bishop1992}. HHT proved the
analogous statement for the maximum of the three
eigenvalues~\cite{HHTsphere}, which is the spherical result used here. The
minimizing partition consists of three equal lunes
(Figure~\ref{fig:partitions}(b); Theorem~\ref{thm:spherical-input}).
Where the projection
$(x,y,z)\mapsto(x,y)$ is non-degenerate, the lune pattern projects onto the
sector pattern. This resemblance does not yield an eigenvalue
comparison: the surface metrics and the radial energy weights differ.

Berkolaiko, Kuchment, and Smilansky related the Morse index of the partition
energy on the space of equipartitions to nodal deficiency, under genericity
assumptions~\cite{BerkolaikoKuchmentSmilansky}. Berkolaiko, Canzani, Cox,
and Marzuola proved minimality of the Y-partition within a homological
class that includes partitions whose interface
contains a curve from the center to the boundary~\cite{Homology}. Berkolaiko,
Canzani, Cox, Kuchment, and Marzuola developed a stability theory relating
the Morse index of the Hessian to nodal deficiency, which yields topologically
constrained minimality in the non-bipartite case~\cite{Stability}. Ognibene
and Velichkov established regularity near triple-junction singularities and
proved the spherical three-partition min--max theorem on $\mathbb S^{d-1}$
for every $d\ge3$~\cite[Theorem~1.1 and Corollary~3.6]{OgnibeneVelichkov}.

The harmonic oscillator has a parallel nodal background. Write
$\lambda_j(H)$ for the eigenvalues of $H=-\Delta+\abs{x}^2$ on the whole
plane, counted with multiplicity. The Hermite basis gives
\[
 \lambda_1(H)=2,\qquad
 \lambda_2(H)=\lambda_3(H)=4,\qquad
 \lambda_4(H)=\lambda_5(H)=\lambda_6(H)=6;
\]
more generally, the level $2(n+1)$ has multiplicity $n+1$;
see~\cite[Section~1]{BerardHelfferOscillator}. The corresponding linear
and bilinear Hermite eigenfunctions produce two half-planes and four
quadrants. Leydold's nodal bound~\cite{Leydold1989} implies that the only
Courant-sharp spectral indices are $1$, $2$, and $4$, as recalled
in~\cite[Theorem~1.2 and the following remark]{BerardHelfferOscillator}.
By contrast, Theorem~\ref{ho:main-partition} places the minimal
three-partition energy at $5$, strictly between $\lambda_3(H)$ and
$\lambda_4(H)$. Since $5$ is not a whole-plane eigenvalue, the minimizing
Y-partition is not a nodal partition of an eigenfunction of $H$.
At high energies, Charron proved a Pleijel-type asymptotic bound on nodal
domain counts for harmonic oscillators in arbitrary dimension, including
anisotropic quadratic traps~\cite{Charron2018}.

The Gaussian ground-state transform~\eqref{ho:gaussian-transform}
relates this problem to Gaussian Dirichlet eigenvalues. Ehrhard proved
that half-spaces minimize the first such eigenvalue at fixed Gaussian
measure~\cite[Theorems~4.7 and~4.10]{Ehrhard1984}. The three-cell Gaussian
optimization considered here, without prescribed cell measures, is treated
in Section~\ref{ho:sec-consequences}.

\subsection{Transplantation methods and the present comparison}
\label{sec:transplantation-background}

Transplantation is a classical method for spectral inequalities.
Conformal transplantation underlies the work of P\'olya and
Szeg\H{o}~\cite{PolyaSzego}, Szeg\H{o}~\cite{Szego1954}, and
Hersch~\cite{Hersch1970}. Hersch's harmonic
transplantation~\cite{Hersch1969} uses Green-function levels;
Bandle, Brillard, and Flucher extended it to spaces of constant
curvature~\cite[Definition~13]{BandleBrillardFlucher}.
Radial, non-conformal constructions include the raywise transformations
of P\'olya and Szeg\H{o}~\cite{PolyaSzego} and the constant-Jacobian
maps of Laugesen and Siudeja~\cite{LaugesenSiudeja}.

The transplantation constructed below differs from these methods in the
quantity it matches.
Its radial coordinate is chosen so that the angular-energy measures of the
two separated model states correspond, rather than areas or Green-function
levels, and a positive multiplier replaces the source radial state by the
spherical one. The resulting map is not conformal: by
Proposition~\ref{prop:stretching}, its radial and angular length factors
differ at every interior radius. What is proved is an inequality between
shifted quadratic forms for a single positive linear operator on the
source form domain, together with a nonnegative defect with a strictly
positive radial weight. Because
one operator serves all three cells, segregation is preserved. Both
applications use the ground-state representation and a one-dimensional
monotone rearrangement. The strict comparison is obtained from a
Sturm--Liouville Wronskian identity for the disc and from a differential
comparison of cumulative densities for the oscillator.

\subsection{Organization}

Section~\ref{sec:transplantation} constructs the disc transplantation,
and Section~\ref{sec:optimality} proves optimality.
Sections~\ref{sec:sphere-input} and~\ref{sec:proofs} establish spherical
saturation and disc rigidity. Section~\ref{sec:general-nu} treats general
angular frequencies. Section~\ref{ho:sec-oscillator} gives the oscillator
application and the common conditional radial-profile principle.

\section{A radial transplantation from the disc to the sphere}
\label{sec:transplantation}

In this section we construct the comparison operator from the two model
ground states. We establish both the quadratic-form inequality used for
optimality in Section~\ref{sec:optimality} and the defect information needed
for the uniqueness argument in Sections~\ref{sec:sphere-input}
and~\ref{sec:proofs}. We use polar coordinates $(r,\phi)$ on $\Disk$ and
colatitude--longitude coordinates $(\theta,\phi)$ on $\Sphere$. The corresponding
energy and mass integrals are
\begin{align*}
 \int_{\Disk}\norm{\nabla u}^2\dd x
 &=\int_0^{2\pi}\int_0^1
 \left(r\abs{u_r}^2+\frac1r\abs{u_\phi}^2\right)\dd r\dd\phi ,\\
 \int_{\Disk}u^2\dd x
 &=\int_0^{2\pi}\int_0^1ru^2\dd r\dd\phi ,
\end{align*}
and
\begin{align*}
 \int_{\Sphere}\norm{\nabla_{\Sphere}v}^2\dd A
 &=\int_0^{2\pi}\int_0^\pi
 \Bigl(\sin\theta\abs{v_\theta}^2+
       \frac{\abs{v_\phi}^2}{\sin\theta}\Bigr)\dd\theta\dd\phi ,\\
 \int_{\Sphere}v^2\dd A
 &=\int_0^{2\pi}\int_0^\pi
       v^2\sin\theta\dd\theta\dd\phi .
\end{align*}
For $0\ne v\in H^1(\Sphere)$,
\[
 \RS(v)=\frac{\int_{\Sphere}\norm{\nabla_{\Sphere}v}^2\dd A}
                  {\int_{\Sphere}v^2\dd A}.
\]
No boundary condition is imposed on the sphere. For an open subset
$E\subset\Sphere$, $H^1_0(E)$ denotes the closure of $C_c^\infty(E)$,
with functions extended by zero into $H^1(\Sphere)$, and $\lambda_1(E)$
is the infimum of $\RS$ over $H^1_0(E)\setminus\set{0}$.

The signs of the shifted quadratic forms
\begin{equation}\label{eq:shifted-forms}
 \QB(u)=\int_{\Disk}\bigl(\norm{\nabla u}^2-Lu^2\bigr)\dd x,
 \qquad
 \QS(v)=\int_{\Sphere}
       \bigl(\norm{\nabla_{\Sphere}v}^2-\mu v^2\bigr)\dd A
\end{equation}
are determined by
\[
 \QB(u)=\bigl(\RB(u)-L\bigr)\norm{u}_{L^2(\Disk)}^2,
 \qquad
 \QS(v)=\bigl(\RS(v)-\mu\bigr)\norm{v}_{L^2(\Sphere)}^2.
\]
We seek a map $T$ with $\QS(Tu)\le c\QB(u)$ for some constant $c>0$: a function
with Rayleigh quotient below $L$ on the disc is then sent to a function with
quotient below $\mu$ on the sphere. For the equality case, we also want the
difference $c\QB(u)-\QS(Tu)$ to control radial variation after factoring out the
disc model state. Since only the signs of the shifted forms matter, $T$ need not
preserve the $L^2$ norm.

\subsection{The separated model states}
\label{sec:radial-models}

Define
\[
 R(r)=J_{3/2}(\kappa r),\qquad S(\theta)=\sin^{3/2}\theta.
\]
Using the angular factor $\sin(\nu(\phi-\phi_0))$ in separation of variables,
we obtain the angular equation $-G''=\nu^2G$ and the following radial equations,
respectively, on the disc and the sphere:
\[
 -F_{\Disk}''-\frac1rF_{\Disk}'+\frac{\nu^2}{r^2}F_{\Disk}=\lambda F_{\Disk},
 \qquad
 -F_{\Sphere}''-\cot\theta F_{\Sphere}'
       +\frac{\nu^2}{\sin^2\theta}F_{\Sphere}=\lambda F_{\Sphere}.
\]
The function $R$ solves the first equation with the Dirichlet condition at
$r=1$ when $\lambda=L=\kappa^2$; the function $S$ solves the second when
$\lambda=\mu$:
\begin{align}
 -R''-\frac1rR'+\frac{\nu^2}{r^2}R&=LR,
             &&0<r<1,\label{eq:R-equation}\\
 -S''-\cot\theta S'+\frac{\nu^2}{\sin^2\theta}S&=\mu S,
             &&0<\theta<\pi.\label{eq:S-equation}
\end{align}
For \eqref{eq:S-equation}, we have $S'/S=\nu\cot\theta$ and
$S''/S=\nu^2\cot^2\theta-\nu\csc^2\theta$. Substitution into the
left-hand side gives $\nu(\nu+1)S$. In divergence form, the two equations
are
\begin{equation}\label{eq:divergence-forms}
 -(rR')'+\Bigl(\frac{\nu^2}{r}-Lr\Bigr)R=0,
 \qquad
 -(\sin\theta S')'+\Bigl(\frac{\nu^2}{\sin\theta}-\mu\sin\theta\Bigr)S=0.
\end{equation}

The function $R$ is positive on $(0,1)$ because $\kappa$ is the first
positive zero of $J_{3/2}$ and the leading term of $R$ at the origin is
positive. Its endpoint expansions are
\begin{align}
 R(r)&=\beta r^{3/2}(1+O(r^2)),
 &\beta&=\frac{(\kappa/2)^{3/2}}{\Gamma(5/2)}>0,
 &&r\downarrow0,\label{eq:R-origin}\\
 R(r)&=\gamma(1-r)(1+O(1-r)),
 &\gamma&=-R'(1)>0,
 &&r\uparrow1.\label{eq:R-boundary}
\end{align}
The first expansion follows from the Bessel series; see~\cite[Eq.~10.2.2]{DLMF}.
At $r=1$ the zero is simple: if both $R(1)$ and
$R'(1)$ vanished, uniqueness for the regular second-order equation
\eqref{eq:R-equation} near $1$ would force $R\equiv0$. Positivity on the
left then gives $R'(1)<0$.

On an angular interval of length $\pi/\nu=2\pi/3$, the functions
\[
 R(r)\sin\bigl(\nu(\phi-\phi_0)\bigr)
 \quad\text{and}\quad
 S(\theta)\sin\bigl(\nu(\phi-\phi_0)\bigr)
\]
are positive Dirichlet eigenfunctions on the corresponding disc sector and
spherical lune, with eigenvalues $L$ and $\mu$. They belong to the corresponding
Dirichlet form domains: they are continuous up to the boundary and vanish
there, and their gradients are $O(\rho^{1/2})$ at distance $\rho$ from a
vertex, hence square integrable. A positive Dirichlet eigenfunction of a
connected domain has positive inner product with a nonnegative first
eigenfunction and is therefore itself a first eigenfunction; thus both
functions are ground states. In particular, the Y-partition of the disc
has energy $L$, and the corresponding lune has first eigenvalue $\mu$.

\subsection{The form of a radial transplantation}
\label{sec:what-is-forced}

The sector and lune ground states have the same angular factor. This suggests
retaining the angular variable and changing only the radial coordinate and the
radial multiplier. We therefore look for a smooth increasing bijection
$\Theta\colon(0,1)\to(0,\pi)$ with $\Theta'>0$, together with a smooth
function $m>0$, such that
\begin{equation}\label{eq:ansatz}
 (Tu)(\Theta(r),\phi)=m(r)u(r,\phi).
\end{equation}
We seek a map from $H^1_0(\Disk)$ to $H^1(\Sphere)$ such that, for a
constant $c>0$ independent of $u$,
\begin{equation}\label{eq:target}
 \QS(Tu)\le c\QB(u),\qquad u\in H^1_0(\Disk),
\end{equation}
with equality for the three sector ground states. A positive multiplier and a
common change of coordinates preserve nonnegativity and segregation. The task
is to determine the multiplier $m$ and coordinate change $\Theta$ so that this
comparison holds.

We first derive necessary conditions for such a map. These conditions specify a
candidate; we then prove its mapping and comparison properties.

\subsubsection*{Sharpness determines the multiplier}
Let $G$ be the function equal to $\sin(\nu(\phi-\phi_0))$ on the angular
interval $(\phi_0,\phi_0+2\pi/3)$ and extended by zero to the rest of the
circle. The function $u=RG$ is a sector ground state, and $\QB(u)=0$.
Under \eqref{eq:ansatz}, its image is nonnegative and vanishes outside
the corresponding spherical lune. For any candidate mapping $H^1_0(\Disk)$
into $H^1(\Sphere)$, the image $Tu$ belongs to the lune's Dirichlet
form domain: the lune has Lipschitz boundary, and the zero exterior value
gives zero trace. Since $\QB(u)=0$, equality in~\eqref{eq:target} would
therefore imply that $Tu$ is a first eigenfunction of that lune. By
simplicity of the first eigenvalue,
\[
 m(r)R(r)=bS(\Theta(r)),
\]
for a constant $b>0$. Multiplying $T$ by a constant changes $c$ by the
corresponding squared factor, so we normalize $b=1$. Thus, once $\Theta$ is
chosen, sharpness fixes the multiplier:
\begin{equation}\label{eq:transplant}
 (Tu)(\Theta(r),\phi)=\frac{S(\Theta(r))}{R(r)}u(r,\phi).
\end{equation}
Writing $\Psi(r,\phi)=(\Theta(r),\phi)$, this says that $Tu=Sz$ when
$u=Rw$ and $z\circ\Psi=w$. In particular, the purely radial function
$R$ is sent to $S$. No separation of variables is imposed on a general
input $u$: the quotient $w=u/R$ may depend on both variables.

\subsubsection*{Factoring out the radial states}
We next calculate the shifted forms in terms of $w$ and $z$. For this
calculation, take
$u\in C_c^\infty(\Disk\setminus\set{0})$. Its support stays away from
$r=0$ and $r=1$, and its image stays away from both spherical poles, so
all radial integrations by parts have zero endpoint terms. The cancellation
is the ground-state representation (also called the ground-state
transform); see, for example,~\cite[Lemma~2.4]{PinchoverTintarev}. We record
and prove its one-dimensional form for the coefficients used here.

\begin{lemma}\label{lem:transform}
Let $J$ be an open interval, and let $a$ and $V$ be smooth on $J$, with $a>0$.
Let $\rho>0$ solve $-(a\rho')'+V\rho=0$ there. Then, for every
$w\in C_c^\infty(J)$,
\[
 \int_J\bigl(a\abs{(\rho w)'}^2+V\rho^2w^2\bigr)\dd t
 =\int_Ja\rho^2\abs{w'}^2\dd t.
\]
\end{lemma}

\begin{proof}
Expand the derivative and integrate the cross term by parts:
\begin{align*}
 a\abs{(\rho w)'}^2
 &=a\rho^2\abs{w'}^2+a(\rho')^2w^2+a\rho\rho'(w^2)',\\
 \int_Ja\rho\rho'(w^2)'\dd t
 &=-\int_J\bigl(a(\rho')^2+\rho(a\rho')'\bigr)w^2\dd t.
\end{align*}
The terms containing $(\rho')^2$ cancel, and
$\rho(a\rho')'=V\rho^2$ cancels the potential term.
\end{proof}

For the disc equation in~\eqref{eq:divergence-forms}, use $a=r$,
$V=\nu^2/r-Lr$, and $\rho=R$. For each fixed $\phi$, the lemma gives
\[
 \int_0^1\bigl(r\abs{\partial_r(Rw)}^2-LrR^2w^2\bigr)\dd r
 =\int_0^1rR^2\abs{w_r}^2\dd r
   -\nu^2\int_0^1\frac{R^2}{r}w^2\dd r.
\]
Adding the angular energy and integrating in $\phi$ gives
\begin{equation}\label{eq:disk-factorization}
 \QB(u)=\int_0^{2\pi}\int_0^1
 \Bigl[rR^2\abs{w_r}^2+
       \frac{R^2}{r}\bigl(\abs{w_\phi}^2-\nu^2w^2\bigr)\Bigr]
 \dd r\dd\phi.
\end{equation}
For the sphere, use $a=\sin\theta$,
$V=\nu^2/\sin\theta-\mu\sin\theta$, and $\rho=S$. The same calculation
gives
\[
 \QS(Sz)=\int_0^{2\pi}\int_0^\pi
 \Bigl[S^2\sin\theta\abs{z_\theta}^2+
       \frac{S^2}{\sin\theta}
       \bigl(\abs{z_\phi}^2-\nu^2z^2\bigr)\Bigr]\dd\theta\dd\phi.
\]
Since $S^2\sin\theta=\sin^4\theta$ and
$S^2/\sin\theta=\sin^2\theta$, the substitution $\theta=\Theta(r)$,
with $z_\theta=w_r/\Theta'$ and $z_\phi=w_\phi$, becomes
\begin{equation}\label{eq:sphere-factorization}
 \QS(Tu)=\int_0^{2\pi}\int_0^1
 \Bigl[\frac{\sin^4\Theta}{\Theta'}\abs{w_r}^2+
       \sin^2\Theta \Theta'
       \bigl(\abs{w_\phi}^2-\nu^2w^2\bigr)\Bigr]\dd r\dd\phi.
\end{equation}
Consequently, for any constant $c$,
\begin{equation}\label{eq:difference-structure}
 c\QB(u)-\QS(Tu)=\int_0^{2\pi}\int_0^1
 \bigl[A(r)\abs{w_r}^2+
       B(r)\bigl(\abs{w_\phi}^2-\nu^2w^2\bigr)\bigr]\dd r\dd\phi,
\end{equation}
where
\begin{equation}\label{eq:comparison-coefficients}
 A=crR^2-\frac{\sin^4\Theta}{\Theta'},
 \qquad B=c\frac{R^2}{r}-\sin^2\Theta \Theta'.
\end{equation}
The radial contribution is nonnegative when $A\ge0$, but the angular
contribution has no fixed sign. We next show that the comparison forces the
coefficient $B$ of the angular term to vanish.

\subsection{Matching the angular energies}
\label{sec:angular-matching}

We first use high angular frequencies to show that $B$ cannot be negative.
To see this, set $w(r,\phi)=\eta(r)\cos(N\phi)$ and $u=Rw$, where
$\eta\in C_c^\infty(0,1)$ and $N$ is a positive integer. The right-hand
side of \eqref{eq:difference-structure} is
\[
 \pi\int_0^1A\abs{\eta'}^2\dd r
 +\pi(N^2-\nu^2)\int_0^1B\eta^2\dd r.
\]
If $B$ were negative somewhere, continuity would allow us to choose
$\eta$ supported where $B<0$. Letting $N\to\infty$ would then
contradict \eqref{eq:target}. Hence every map satisfying \eqref{eq:target} must have
$B\ge0$ on $(0,1)$.

A radial test then shows that the total angular surplus cannot be positive. Set
\[
 I=\int_0^1\frac{R(r)^2}{r}\dd r.
\]
The endpoint expansions \eqref{eq:R-origin}--\eqref{eq:R-boundary}
show that $0<I<\infty$. Since $\Theta$ increases from $0$ to $\pi$,
\begin{equation}\label{eq:angular-surplus}
 \int_0^1B(r)\dd r
 =cI-\int_0^\pi\sin^2\theta\dd\theta
 =cI-\frac\pi2.
\end{equation}
In particular, $B\ge0$ implies $cI\ge\pi/2$.

Now test the desired inequality with the purely radial function $u=R$.
The endpoint estimates show that $R\in H^1_0(\Disk)$ and
$S\in H^1(\Sphere)$, and the prescribed multiplier gives $TR=S$.
Multiplying \eqref{eq:R-equation} by $rR$ and integrating by parts gives
\begin{equation}\label{eq:radial-integrated-identity}
 \int_0^1r(R')^2\dd r-L\int_0^1rR^2\dd r=-\nu^2I.
\end{equation}
Here $[rRR']_0^1=0$: at the origin it is $O(r^3)$, and at the boundary
$R(1)=0$. Similarly, multiplying \eqref{eq:S-equation} by
$S\sin\theta$ gives
\[
 \int_0^\pi\sin\theta (S')^2\dd\theta
 -\mu\int_0^\pi\sin\theta S^2\dd\theta
 =-\nu^2\int_0^\pi\frac{S^2}{\sin\theta}\dd\theta
 =-\nu^2\frac\pi2.
\]
The boundary term vanishes because
$\sin\theta SS'=\nu\sin^{2\nu}\theta\cos\theta\to0$ at both
poles. After integration over the angular variable, the test $u=R$
therefore gives
\[
 0\le c\QB(R)-\QS(S)
   =2\pi\nu^2\left(\frac\pi2-cI\right).
\]
Thus $cI\le\pi/2$ as well. Together with
\eqref{eq:angular-surplus}, we obtain
\[
 c=\frac{\pi}{2I},\qquad \int_0^1B(r)\dd r=0.
\]
Since a continuous nonnegative function with zero integral vanishes
identically, the angular coefficients must match exactly:
\begin{equation}\label{eq:theta-derivative}
 \sin^2\Theta(r) \Theta'(r)=c\frac{R(r)^2}{r},
 \qquad 0<r<1.
\end{equation}

Integrating \eqref{eq:theta-derivative} from $0$ to $r$ gives
\begin{equation}\label{eq:cumulative-map}
 \int_0^{\Theta(r)}\sin^2t\dd t
   =c\int_0^r\frac{R(t)^2}{t}\dd t,
 \qquad c=\frac{\pi}{2I}.
\end{equation}
We now use this equation to \emph{define} $\Theta$, independently of the
comparison. Both cumulative functions are smooth and strictly increasing
on their respective intervals, with the same endpoint values $0$
and $\pi/2$. Thus \eqref{eq:cumulative-map} defines a unique smooth
increasing bijection $\Theta\colon(0,1)\to(0,\pi)$, with $\Theta'>0$
and endpoint limits $0$ and $\pi$. Equation~\eqref{eq:theta-derivative}
follows by differentiation.

In other words, $\Theta$ transports the
angular energy measure $cR^2/r\dd r$ of the disc model onto the angular
energy measure $\sin^2\theta\dd\theta$ of the spherical model. After
normalizing by their common total mass, this is the one-dimensional
increasing rearrangement obtained by matching cumulative distribution
functions; see~\cite[Section~2.1]{Santambrogio}. The map
$\Psi(r,\phi)=(\Theta(r),\phi)$ is a diffeomorphism from the punctured
disc onto the sphere with its poles removed. As $r\downarrow0$, $\Psi$
approaches the north pole, and as $r\uparrow1$, it approaches the south pole.
Thus \eqref{eq:transplant} defines $T$ on
$C_c^\infty(\Disk\setminus\set{0})$, with smooth images supported away
from both poles. We next verify the comparison for this candidate and extend
it to the full Sobolev space.

For our choice of $\Theta$, the angular term in
\eqref{eq:difference-structure} vanishes. We therefore obtain the exact
identity
\begin{equation}\label{eq:defect-identity}
 c\QB(u)-\QS(Tu)
 =\int_0^{2\pi}\int_0^1
 A(r)\abs{\partial_r (u/R)}^2\dd r\dd\phi,
 \qquad u\in C_c^\infty(\Disk\setminus\set{0}).
\end{equation}
Using \eqref{eq:theta-derivative}, the coefficient $A$ can be
written in either of the following forms:
\begin{equation}\label{eq:defect-weight}
 A(r)
 =crR(r)^2
       \biggl(1-\frac{\sin^6\Theta(r)}{c^2R(r)^4}\biggr)
 =\frac{\sin^2\Theta(r)}{\Theta'(r)}
       \left(r^2\Theta'(r)^2-\sin^2\Theta(r)\right).
\end{equation}
Indeed,
$\sin^4\Theta/\Theta'=r\sin^6\Theta/(cR^2)$ gives the first
expression directly from \eqref{eq:comparison-coefficients}. Thus it suffices
to prove $A>0$, or equivalently
\begin{equation}\label{eq:stretching-goal}
 r\Theta'(r)>\sin\Theta(r),
 \qquad\text{equivalently}\qquad
 cR(r)^2>\sin^3\Theta(r),\quad 0<r<1.
\end{equation}

Geometrically, under $\Psi$ the spherical
metric becomes $\Theta'(r)^2\dd r^2+\sin^2\Theta(r)\dd\phi^2$. Relative to
the Euclidean polar metric $\dd r^2+r^2\dd\phi^2$, the radial and angular
length factors are $\Theta'$ and $\sin\Theta/r$. A conformal radial map has
equal factors, whereas \eqref{eq:stretching-goal} asserts that the radial
factor is strictly larger. Proposition~\ref{prop:stretching} will therefore
show that $\Psi$ is nonconformal. The angular coefficients match because of
the choice~\eqref{eq:cumulative-map}, not by conformal invariance.

\subsection{A strict radial stretching inequality}
\label{sec:stretching}

The increasing solutions of $r\vartheta'=\sin\vartheta$ with
$\vartheta(0+)=0$ are
\[
 \vartheta_a(r)=2\arctan(ar),\qquad a>0,\quad r>0.
\]
These are the radial parts of inverse stereographic projection composed
with dilations of the plane; see~\cite[Section~2]{BandleBrillardFlucher}.
The same M\"obius transformations appear in Hersch's centering
argument~\cite{Hersch1970}. Here they serve only as comparison maps. They
map $(0,\infty)$ onto $(0,\pi)$ and satisfy
\begin{equation}\label{eq:comparison-map-derivative}
 r\vartheta_a'(r)=\sin\vartheta_a(r)
                 =\frac{2ar}{1+a^2r^2}.
\end{equation}
For each $r_0\in(0,1)$, one such map passes through
$(r_0,\Theta(r_0))$. We compare the associated radial states at this point.
Pulled back by $\vartheta_a$, the spherical state satisfies the disc radial
equation with a strictly decreasing potential in place of the constant $L$.

Put
\[
 Y_a(r)=S(\vartheta_a(r))
       =\left(\frac{2ar}{1+a^2r^2}\right)^\nu,
 \qquad \nu=\frac32.
\]
Since $rY_a'=\nu\cos\vartheta_a Y_a$, a further differentiation gives
\begin{equation}\label{eq:comparison-state-equation}
 -(rY_a')'+\frac{\nu^2}{r}Y_a
   =V_a(r) rY_a,
 \qquad
 V_a(r)=\frac{4\mu a^2}{(1+a^2r^2)^2},
 \qquad \mu=\nu(\nu+1).
\end{equation}
Indeed, $(rY_a')'=(\nu^2-\mu\sin^2\vartheta_a)Y_a/r$.
The potential $V_a$ is strictly decreasing for $r>0$.
Equation~\eqref{eq:comparison-map-derivative} also gives the normalization
\begin{equation}\label{eq:equal-total-masses}
 \int_0^\infty\frac{Y_a(r)^2}{r}\dd r
 =\int_0^\pi\sin^2\theta\dd\theta
 =\frac\pi2
 =c\int_0^1\frac{R(r)^2}{r}\dd r.
\end{equation}
Here the change of variables is $\theta=\vartheta_a(r)$, for which
$\dd r/r=\dd\theta/\sin\theta$.

\begin{proposition}\label{prop:stretching}
The map $\Theta$ defined by \eqref{eq:cumulative-map} satisfies
\eqref{eq:stretching-goal} at every $r\in(0,1)$.
\end{proposition}

\begin{proof}
Fix $a>0$. We first show that the ratio $R/Y_a$ is unimodal: it is either
strictly decreasing, or it increases strictly up to one point and then
decreases strictly. Consider the Wronskian
\[
 W_a(r)=r\bigl(R'(r)Y_a(r)-R(r)Y_a'(r)\bigr).
\]
Subtracting the disc equation \eqref{eq:R-equation} and
\eqref{eq:comparison-state-equation} in Wronskian form is the usual
Sturm--Liouville Lagrange-identity calculation;
see~\cite[Sections~5.4--5.5]{TeschlODE}. In the present notation it gives
\begin{equation}\label{eq:comparison-wronskian}
 W_a'(r)=\bigl(V_a(r)-L\bigr)rR(r)Y_a(r),
 \qquad
 \left(\frac{R}{Y_a}\right)'=\frac{W_a}{rY_a^2}.
\end{equation}
The origin expansions of $R$ and $Y_a$ imply $W_a(0)=0$, where $W_a(0)$
denotes the right-hand limit at $0$. At the other endpoint, we have
\[
 W_a(1)=R'(1)Y_a(1)<0.
\]
On $(0,1)$ the factor $rRY_a$ is positive, whereas $V_a-L$ is
strictly decreasing. If $V_a(0+)\le L$, then $W_a'<0$ throughout
$(0,1)$, so $W_a<0$. If $V_a(0+)>L$, then $W_a$ initially increases.
Since $W_a(1)<0$, the sign change of $V_a-L$ occurs inside $(0,1)$:
otherwise $W_a$ would increase on all of $(0,1)$ and be positive at $1$.
Hence $W_a$ increases from $0$, then decreases to a negative value,
and vanishes exactly once. The second identity in
\eqref{eq:comparison-wronskian} proves the asserted unimodality. In
particular, every strict superlevel set of $R/Y_a$ is an interval, and
$R/Y_a$ is not constant on any open interval.

Now fix $r_0\in(0,1)$ and choose
\[
 a=\frac{\tan(\Theta(r_0)/2)}{r_0},
 \qquad\text{so that}\qquad
 \vartheta_a(r_0)=\Theta(r_0).
\]
The definition of $\Theta$ and the substitution
$\theta=\vartheta_a(r)$ yield
\[
 c\int_0^{r_0}\frac{R(r)^2}{r}\dd r
 =\int_0^{\Theta(r_0)}\sin^2\theta\dd\theta
 =\int_0^{r_0}\frac{Y_a(r)^2}{r}\dd r.
\]
Together with \eqref{eq:equal-total-masses}, this gives
\begin{align}
 \int_0^{r_0}\frac{cR^2-Y_a^2}{r}\dd r&=0,
 \label{eq:two-mass-balances}\\
 \int_{r_0}^1\frac{cR^2-Y_a^2}{r}\dd r
 &=\int_1^\infty\frac{Y_a^2}{r}\dd r>0.
 \label{eq:comparison-tail}
\end{align}
The first balance forces $\sqrt{c} R/Y_a>1$ somewhere in $(0,r_0)$.
Otherwise, the integrand would be nonpositive there, and the zero
integral would force $R/Y_a=1/\sqrt{c}$ throughout $(0,r_0)$, contrary to
the strict unimodality just proved. The second balance forces
$\sqrt{c} R/Y_a>1$ somewhere in $(r_0,1)$.
Since the strict superlevel set $\set{\sqrt{c} R/Y_a>1}$ is an interval,
it contains $r_0$. Hence
\[
 cR(r_0)^2>Y_a(r_0)^2=\sin^3\Theta(r_0).
\]
Finally, \eqref{eq:theta-derivative} gives
$r_0\Theta'(r_0)>\sin\Theta(r_0)$. Since $r_0$ was arbitrary,
this proves the proposition.
\end{proof}

\begin{proposition}\label{prop:core-comparison}
For $u\in C_c^\infty(\Disk\setminus\set{0})$, the defect identity
\eqref{eq:defect-identity} holds with the weight in
\eqref{eq:defect-weight}, and $A(r)>0$ for every $0<r<1$. In particular,
\begin{equation}\label{eq:core-inequality}
 \QS(Tu)\le c\QB(u).
\end{equation}
\end{proposition}

\begin{proof}
The identity was derived in Section~\ref{sec:angular-matching}.
Proposition~\ref{prop:stretching} and either expression in
\eqref{eq:defect-weight} give $A>0$, and the comparison follows.
\end{proof}

\subsection{The estimates needed at the poles}
\label{sec:sobolev}

We now extend the comparison from
smooth test functions to the full Dirichlet form domain. The extension
applies to functions supported in arbitrary open cells and carries the
defect estimate needed in the equality case.
We first prove a norm bound and then use density to obtain the extension.

\begin{lemma}\label{lem:weight-bound}
The function $r\mapsto\sin\Theta(r)/r$ is bounded on $(0,1)$.
\end{lemma}

\begin{proof}
By Proposition~\ref{prop:stretching}, the logarithmic derivative of
$r\mapsto\tan(\Theta(r)/2)/r$ is positive, so this function is increasing. For $0<r\le\frac12$
this yields
$\sin\Theta(r)\le2\tan(\Theta(r)/2)\le4r\tan(\Theta(\frac12)/2)$, and for
$\frac12\le r<1$, we have $\sin\Theta(r)/r\le2$.
\end{proof}

For a smooth core function, a change of variables and
\eqref{eq:theta-derivative} give
\begin{align}
 \norm{Tu}_{L^2(\Sphere)}^2
 &=\int_0^{2\pi}\int_0^1
     \frac{\sin^4\Theta(r)}{R(r)^2}\Theta'(r)\abs{u(r,\phi)}^2\dd r\dd\phi
     \notag\\
 &=c\int_0^{2\pi}\int_0^1
       \frac{\sin^2\Theta(r)}{r}\abs{u(r,\phi)}^2\dd r\dd\phi
     \notag\\
 &=c\int_{\Disk}
       \left(\frac{\sin\Theta(\norm{x})}{\norm{x}}\right)^2\abs{u(x)}^2\dd x
     \notag\\
 &\le C_0\norm{u}_{L^2(\Disk)}^2,
\label{eq:mass-identity}
\end{align}
where $C_0=c\sup_{0<r<1}(\sin\Theta(r)/r)^2<\infty$ by
Lemma~\ref{lem:weight-bound}. Hence $T$ is bounded on $L^2$. The weight is
positive for $0<\norm{x}<1$, so $T$ is injective, but this weight tends to zero as
$\norm{x}\uparrow1$, so the source and target $L^2$ norms are not
equivalent.

Combining \eqref{eq:mass-identity} with \eqref{eq:core-inequality} and
$\QB(u)\le\norm{\nabla u}_{L^2(\Disk)}^2$ (as $L>0$) yields
\begin{align}
 \norm{Tu}_{H^1(\Sphere)}^2
 &=\QS(Tu)+(\mu+1)\norm{Tu}_{L^2(\Sphere)}^2\notag\\
 &\le c\QB(u)+(\mu+1)C_0\norm{u}_{L^2(\Disk)}^2\notag\\
 &\le c\norm{\nabla u}_{L^2(\Disk)}^2
       +(\mu+1)C_0\norm{u}_{L^2(\Disk)}^2.
 \label{eq:H1-bound}
\end{align}
Although the shifted form $\QS(Tu)$ need not be nonnegative, the resulting
estimate is an ordinary $H^1$ bound for $T$ on the smooth core.

\subsubsection*{A dense core avoiding the center}

To extend $T$, we use the fact that an interior point is removable
for planar Dirichlet Sobolev spaces. We state it for arbitrary open cells,
since the extension must also preserve each cell's Dirichlet form domain.

\begin{lemma}\label{lem:punctured-core}
Let $D\subset\R^2$ be open and let $p\in\R^2$. Then
\begin{equation}\label{eq:point-removability}
 H^1_0(D\setminus\set{p})=H^1_0(D),
\end{equation}
where both spaces are identified by zero extension to $\R^2$.
The same conclusion holds for the closures of $C_c^\infty(D\setminus\set{p})$
and $C_c^\infty(D)$ in the norm
 $\bigl(\int_{\R^2}(\abs{\nabla u}^2+(1+V)\abs{u}^2)\dd x\bigr)^{1/2}$,
whenever $V\ge0$ is locally bounded.
\end{lemma}

\begin{proof}
By translation, take $p=0$. For $0<\varepsilon<1$, set
\[
 \chi_\varepsilon(r)=
 \begin{cases}
 0,&r\le\varepsilon^2,\\
 \dfrac{\log(r/\varepsilon^2)}{\log(1/\varepsilon)},
       &\varepsilon^2<r<\varepsilon,\\
 1,&r\ge\varepsilon.
 \end{cases}
 \qquad
 \int_{\R^2}\abs{\nabla\chi_\varepsilon}^2\dd x
       =\frac{2\pi}{\log(1/\varepsilon)}\to 0.
\]
For $\varphi\in C_c^\infty(D)$, boundedness of $\varphi$ and dominated
convergence give $\chi_\varepsilon\varphi\to\varphi$ in $H^1$.
The potential term also converges, since $V$ is bounded on the fixed
compact support. Mollifying $\chi_\varepsilon\varphi$ inside
$D\setminus\set{0}$ gives approximants in $C_c^\infty(D\setminus\set{0})$.
Density and the reverse inclusion prove both assertions.
\end{proof}

\subsection{The full-space comparison and its equality case}
\label{sec:global}

We collect the properties used in the partition argument: the form
inequality, the formula for $T$, the form-domain membership of functions
supported in a cell, and the strict defect that rules out radial variation.
To pass to the limit in the defect identity, we first restrict the integral
to compact annuli, where no endpoint singularity is present.

\begin{proposition}\label{prop:full-comparison}
Formula~\eqref{eq:transplant} extends to a bounded linear operator
\[
 T\colon H^1_0(\Disk)\to H^1(\Sphere),
\]
with the following properties.
\begin{enumerate}
\item\label{item:formula}
The extension is given by the same formula almost everywhere. It is
injective and preserves nonnegativity and segregation.
\item\label{item:inequality}
For every $u\in H^1_0(\Disk)$,
\begin{equation}\label{eq:global-comparison}
 \QS(Tu)\le c\QB(u),
\end{equation}
and for every $0<a<b<1$,
\begin{equation}\label{eq:localized-defect}
 c\QB(u)-\QS(Tu)
 \ge\int_0^{2\pi}\int_a^b
 A(r)\abs{\partial_r(u/R)}^2\dd r\dd\phi .
\end{equation}
\item\label{item:cells}
If $D\subset\Disk$ is open and $u\in H^1_0(D)$, then
$Tu\in H^1_0\bigl(\Psi(D\setminus\set{0})\bigr)$.
\item\label{item:rigidity}
If $c\QB(u)=\QS(Tu)$, there exists $f\in H^1(\Circle)$ such that
\begin{equation}\label{eq:radial-rigidity}
 u(r,\phi)=R(r)f(\phi)\quad\text{a.e. on }\Disk.
\end{equation}
\end{enumerate}
\end{proposition}

\begin{proof}
The estimate \eqref{eq:H1-bound} and Lemma~\ref{lem:punctured-core} with
$D=\Disk$ give the required boundedness and a unique linear extension. Since
$T$ is also bounded from $L^2(\Disk)$ to $L^2(\Sphere)$ by
\eqref{eq:mass-identity}, the extension is given by the same formula
almost everywhere, and \eqref{eq:mass-identity} continues to hold. Writing
$m(r)=S(\Theta(r))/R(r)$, the formula is $Tu=(m u)\circ\Psi^{-1}$.
The diffeomorphism $\Psi$ between the punctured disc and the sphere minus
the poles preserves null sets in both directions, and $m$ is positive.
Thus $T$ preserves nonnegativity and segregation; it is injective because
the weight in~\eqref{eq:mass-identity} is positive almost everywhere.
This proves item~\ref{item:formula}.

The two quadratic forms are continuous in their respective $H^1$ norms,
so \eqref{eq:global-comparison} follows from \eqref{eq:core-inequality}
by taking limits. Let $u_n\in C_c^\infty(\Disk\setminus\set{0})$
converge to $u$ in $H^1_0(\Disk)$. On the annulus $a<r<b$, the function
$R$ is smooth and bounded away from zero, so $u_n/R\to u/R$ strongly in
$H^1$ there; polar and Cartesian Sobolev norms are equivalent on the
annulus, and $A$ is bounded on $[a,b]$. Restricting the defect integral in
\eqref{eq:defect-identity} to the annulus and passing to the limit, while
discarding the nonnegative contribution from the rest of the disc, proves
\eqref{eq:localized-defect}.
This proves item~\ref{item:inequality}.

For item~\ref{item:cells}, Lemma~\ref{lem:punctured-core} gives
$u_n\in C_c^\infty(D\setminus\set{0})$ with $u_n\to u$ in $H^1_0(\Disk)$.
Each $Tu_n$ is smooth with compact support in $\Psi(D\setminus\set{0})$,
and $Tu_n\to Tu$ in $H^1(\Sphere)$. Hence
$Tu\in H^1_0(\Psi(D\setminus\set{0}))$.

For item~\ref{item:rigidity}, suppose that the global defect
$c\QB(u)-\QS(Tu)$ vanishes. The localized inequality and the strict positivity
of $A$ on every compact subinterval then give
\[
 \partial_r(u/R)=0\quad\text{a.e. on every compact annulus.}
\]
On a product cylinder $(a,b)\times\Circle$, Sobolev slicing
\cite[Chapter~4]{EvansGariepy2015} gives $u/R=f_{a,b}(\phi)$ almost
everywhere. The radial average
\[
 f_{a,b}(\phi)=\frac{1}{b-a}\int_a^b\frac{u(r,\phi)}{R(r)}\dd r
\]
belongs to $H^1(\Circle)$: its weak angular derivative is the average
of $\partial_\phi(u/R)$. Overlapping annuli identify these factors and
give a single $f\in H^1(\Circle)$ satisfying~\eqref{eq:radial-rigidity}.
This proves item~\ref{item:rigidity} and completes the proof.
\end{proof}

\section{Optimality of the Y-partition}\label{sec:optimality}

We use the comparison from Section~\ref{sec:transplantation} to determine the
optimal value. We first recall the spherical result and justify its use for
arbitrary open cells.

An open three-partition of $\Sphere$ is a triple of pairwise disjoint,
nonempty, connected open subsets of $\Sphere$. Its energy and the minimal
energy $\partenergy_3(\Sphere)$ are defined as in
\eqref{eq:partition-energy}, with $\lambda_1(E)$ as in
Section~\ref{sec:transplantation}.

\begin{theorem}[Helffer, Hoffmann-Ostenhof, Terracini]
\label{thm:spherical-input}
The minimal three-partition energy of the sphere is
$\partenergy_3(\Sphere)=\mu=15/4$. The value is attained by a regular strong
partition, and every regular strong minimizing partition consists, up to a
rotation and relabeling, of the three equal spherical lunes.
\end{theorem}

The value and the classification of regular minimizing partitions
are proved in~\cite[Theorem~1.1]{HHTsphere}. Their partition class, defined
on p.~154 immediately before (1.2), consists of regular open domains.
For an arbitrary open triple $(E_i)$ and
$\varepsilon>0$, compactly supported test functions give smooth connected
subdomains $F_i\Subset E_i$ with
$\lambda_1(F_i)\le\lambda_1(E_i)+\varepsilon$.
Applying the regular theorem to $(F_i)$ and letting $\varepsilon\downarrow0$
gives the same lower bound for $(E_i)$. The equal lunes attain $\mu$.

\begin{proof}[Proof of optimality in Theorem~\ref{thm:main}]
The sector ground states constructed in Section~\ref{sec:radial-models} give
$\partenergy_3(\Disk)\le L$. Suppose, for a contradiction, that an open
three-partition $(D_1,D_2,D_3)$ satisfies
\[
 \max_i\lambda_1(D_i)<L.
\]
By the variational characterization of the first eigenvalue, choose
nonzero functions $u_i\in H^1_0(D_i)$ with $\RB(u_i)<L$, and set
\[
 E_i=\Psi(D_i\setminus\set{0}).
\]
Removing a point from a connected open subset of the plane leaves it
connected, so the $E_i$ are nonempty, connected, pairwise disjoint open
subsets of the sphere. By Proposition~\ref{prop:full-comparison}, each
$Tu_i$ is nonzero and belongs to $H^1_0(E_i)$, and
\[
 \QS(Tu_i)\le c\QB(u_i)
   =c\bigl(\RB(u_i)-L\bigr)\norm{u_i}_{L^2(\Disk)}^2<0.
\]
Consequently,
\[
 \lambda_1(E_i)\le\RS(Tu_i)<\mu,
 \qquad i=1,2,3,
\]
contradicting Theorem~\ref{thm:spherical-input}. Hence
$\partenergy_3(\Disk)\ge L$. Together with the upper bound, this proves
$\partenergy_3(\Disk)=L$, attained by the Y-partition.
\end{proof}

\section{Spherical saturation}\label{sec:sphere-input}

The spherical lower bound rules out three Rayleigh quotients that are all
strictly below $\mu$, which was enough for optimality. To identify the
equality case, we need the stronger implication
\[
 \RS(v_i)\le\mu\quad(i=1,2,3)
 \implies
 \RS(v_1)=\RS(v_2)=\RS(v_3)=\mu
\]
for nonnegative, nonzero, segregated functions $v_1,v_2,v_3$ in
$H^1(\Sphere)$. The lower bound for the maximum alone does not exclude one
quotient below $\mu$ while the other two equal $\mu$. We deduce this
saturation statement from the equipartition theorem for minimal
spectral partitions. Since the positivity sets of the $v_i$ need only be
measurable, the theorem must be used in the measurable-support formulation.

\subsection{The functional lower bound}

The lower bound for segregated functions follows from the open-partition
value alone. We record this direct argument separately from the stronger
saturation statement: the lower bound suffices for optimality in both
planar problems, whereas their rigidity proofs also use saturation.

\begin{lemma}\label{lem:spherical-functional}
Let $g_1,g_2,g_3\in H^1(\Sphere)$ be nonnegative and nonzero, with
$g_i g_j=0$ almost everywhere for $i\ne j$.  Then
\begin{equation}\label{ho:sphere-functional}
 \max_i\frac{\int_{\Sphere}\abs{\nabla_{\Sphere}g_i}^2\dd A}
                  {\int_{\Sphere}\abs{g_i}^2\dd A}
 \ge\mu.
\end{equation}
Equivalently, at least one of the numbers $\QS(g_i)$ is nonnegative.
\end{lemma}

\begin{proof}
Choose smooth approximations to the $g_i$ in $H^1(\Sphere)$ and take their positive
parts, obtaining nonnegative continuous functions $b_{i,n}\to g_i$ in
$H^1(\Sphere)$.  Set
\[
 w_{i,n}=\Bigl(b_{i,n}-\sum_{j\ne i}b_{j,n}\Bigr)_+.
\]
The $w_{i,n}$ are continuous and have pairwise disjoint positivity sets.
The continuity of the positive-part operation in $H^1$ implies
$w_{i,n}\to g_i$ in $H^1$, since
$(g_i-\sum_{j\ne i}g_j)_+=g_i$.

Let $U_{i,n}=\{w_{i,n}>0\}$.  Then $w_{i,n}\in H_0^1(U_{i,n})$:
indeed, $(w_{i,n}-\varepsilon)_+$ has compact support in $U_{i,n}$ and
converges to $w_{i,n}$ in $H^1$ as $\varepsilon\downarrow0$.
A function in $H^1$ with compact support in an open set can be approximated
there by smooth compactly supported functions, so these truncations
belong to $H^1_0(U_{i,n})$.
If all three quotients in \eqref{ho:sphere-functional} were less than
$\mu$, the same would hold for $w_{i,n}$ for large $n$.  Thus the three
disjoint open sets $U_{i,n}$ would have first eigenvalues below $\mu$.
When one of these sets is disconnected, the decomposition of the
Dirichlet form over its components supplies a connected component with
first eigenvalue below $\mu$.  This contradicts
Theorem~\ref{thm:spherical-input}.
\end{proof}

\subsection{Measurable supports and equipartition}

For a measurable $A\subset\Sphere$, set
\begin{equation}\label{eq:spherical-relaxed-space}
 \mathcal H(A)=\set{w\in H^1(\Sphere)\colon w=0\text{ a.e. on }\Sphere\setminus A},
 \qquad
 \widehat\lambda_1(A)=
 \inf_{0\ne w\in\mathcal H(A)}\RS(w),
\end{equation}
with $\widehat\lambda_1(A)=+\infty$ when $\mathcal H(A)=\set{0}$.
The space $\mathcal H(A)$ is a closed subspace of $H^1(\Sphere)$.
Even for an open set $A$, this space need not equal $H^1_0(A)$:
removing a smooth interior arc of zero area does not change an
almost-everywhere support condition, but the approximations defining
$H^1_0(A)$ impose zero trace on that arc.

For a measurably disjoint triple $\mathcal A=(A_1,A_2,A_3)$, meaning
$\measure{A_i\cap A_j}=0$ for $i\ne j$, define
\[
 \widehat\Lambda(\mathcal A)=\max_i\widehat\lambda_1(A_i),\qquad
 \widehat{\partenergy}_3(\Sphere)=\inf_{\mathcal A}\widehat\Lambda(\mathcal A).
\]
Here $\measure{\cdot}$ denotes surface area, and the sets need not cover
the sphere. This is the measurable-support relaxation
of~\cite[Definition~3.1 and Remark~3.3, pp.~109--110]{HHTnodal}.

\begin{theorem}[Relaxed spherical optimum and equipartition]
\label{thm:spherical-equipartition}
The measurable-support and open-partition optima agree:
\begin{equation}\label{eq:spherical-relaxed-input}
 \widehat{\partenergy}_3(\Sphere)=\partenergy_3(\Sphere)=\mu.
\end{equation}
Every minimizing measurable triple is an equipartition:
\begin{equation}\label{eq:relaxed-equipartition}
 \widehat\Lambda(\mathcal A)=\mu
 \implies
 \widehat\lambda_1(A_i)=\mu\quad(i=1,2,3).
\end{equation}
\end{theorem}

\begin{proof}
If $\widehat\Lambda(\mathcal A)<\infty$, compactness of
$H^1(\Sphere)\hookrightarrow L^2(\Sphere)$ gives a nonnegative normalized
ground state $\varphi_i\in\mathcal H(A_i)$ for every $i$.
These states are segregated, so Lemma~\ref{lem:spherical-functional}
gives $\widehat\Lambda(\mathcal A)\ge\mu$. The three equal lunes attain
$\mu$, proving~\eqref{eq:spherical-relaxed-input}.

Now let $\mathcal A$ be any minimizing measurable triple and choose such
$\varphi_i$. We use the variational inequalities
of~\cite[Theorem~3.4]{HHTnodal}, applied to $-\Delta_{\Sphere}+1$
to meet the positive-operator convention of its Remark~3.2.
The shift leaves the minimizers and ground states unchanged and changes
the optimal value from $\mu$ to $\mu+1$.
The proof uses compact Sobolev embeddings and variations of functions
and applies with the spherical gradient and area measure; this spherical
application is also explicit in the proof of~\cite[Corollary~3.6]{OgnibeneVelichkov}.
Thus there is a multiplier vector
$a=(a_1,a_2,a_3)\in[0,\infty)^3\setminus\{0\}$ such that
$z_i=a_i\varphi_i$ satisfy, after cancelling the added zeroth-order
terms on both sides,
\[
 -\Delta_{\Sphere}z_i\le\mu z_i,\qquad
 -\Delta_{\Sphere}\Bigl(z_i-\sum_{j\ne i}z_j\Bigr)
 \ge\mu\Bigl(z_i-\sum_{j\ne i}z_j\Bigr)
\]
in the distributional sense on $\Sphere$.

All components of $a$ are nonzero. Indeed, if $a_3=0$ after relabeling,
the second inequalities for $i=1,2$ give
$-\Delta_{\Sphere}(z_1-z_2)=\mu(z_1-z_2)$.
Segregation and $a\ne0$ imply that $z_1-z_2$ is nonzero, but it vanishes
almost everywhere on $A_3$, which has positive area. This contradicts
unique continuation for spherical eigenfunctions, as in the cited proof
of Corollary~3.6.

Finally, test both inequalities with $z_i$. Segregation gives
$z_i z_j=0$ and $\nabla_{\Sphere}z_i\cdot\nabla_{\Sphere}z_j=0$
almost everywhere for $i\ne j$, so the two inequalities yield
$\int_{\Sphere}\abs{\nabla_{\Sphere}z_i}^2\dd A
=\mu\int_{\Sphere}z_i^2\dd A$.
This is the equality for nonzero multipliers noted
in~\cite[Remark~3.14(d)]{HHTnodal}. Since $a_i>0$, it gives
$\widehat\lambda_1(A_i)=\RS(\varphi_i)=\mu$ for every $i$.
\end{proof}

\subsection{Saturation for segregated functions}

Equipartition concerns the lowest eigenvalues of the supports, whereas
we need a statement about the Rayleigh quotients of the original
functions. The following variational argument connects them. The functions
need not be ground states.

\begin{proposition}\label{prop:saturation}
Let $v_1$, $v_2$, $v_3$ be nonnegative, nonzero, pairwise segregated
functions in $H^1(\Sphere)$. Then
\begin{equation}\label{eq:spherical-functional-lower}
 \max_i\RS(v_i)\ge\mu.
\end{equation}
If $\RS(v_i)\le\mu$ for all three indices, then
\[
 \RS(v_1)=\RS(v_2)=\RS(v_3)=\mu.
\]
Equivalently, the three shifted forms cannot all be negative, and if they
are all nonpositive, they are all zero.
\end{proposition}

\begin{proof}
Choose measurable representatives and set $A_i=\set{v_i>0}$. Segregation
makes these sets measurably disjoint, and $v_i\in\mathcal H(A_i)$ gives
\[
 \widehat\lambda_1(A_i)\le\RS(v_i).
\]
By \eqref{eq:spherical-relaxed-input},
\[
 \mu\le\max_i\widehat\lambda_1(A_i)\le\max_i\RS(v_i),
\]
which proves \eqref{eq:spherical-functional-lower}.

If every $\RS(v_i)\le\mu$, the same chain shows that
$(A_1,A_2,A_3)$ is a relaxed minimizing partition. Equipartition
\eqref{eq:relaxed-equipartition} then gives
$\widehat\lambda_1(A_i)=\mu$ for each $i$. Consequently,
\[
 \mu=\widehat\lambda_1(A_i)\le\RS(v_i)\le\mu,
\]
and every quotient equals $\mu$.
\end{proof}

\section{Equality and uniqueness}
\label{sec:proofs}

Optimality for open partitions was proved in Section~\ref{sec:optimality}.
We now prove the stronger functional statement of Theorem~\ref{thm:functional}
and use its equality case to classify all minimizing open partitions.

\subsection{The functional lower bound and equality of the forms}

Let $(u_1,u_2,u_3)$ be a segregated triple in $H^1_0(\Disk)$. By
Proposition~\ref{prop:full-comparison}, the functions $Tu_i$ are
nonnegative, nonzero, and segregated in $H^1(\Sphere)$, and
\begin{equation}\label{eq:transplanted-chain}
 \QS(Tu_i)\le c\QB(u_i)
  =c\bigl(\RB(u_i)-L\bigr)\norm{u_i}_{L^2(\Disk)}^2 ,
 \qquad i=1,2,3.
\end{equation}

If $\max_i\RB(u_i)<L$, then $\QS(Tu_i)<0$ for each $i$, that is,
$\RS(Tu_i)<\mu$, contradicting \eqref{eq:spherical-functional-lower}. This
proves \eqref{eq:functional-bound}. The three sector ground states
\eqref{eq:sector-groundstates} form a segregated triple with Rayleigh
quotient $L$, so the bound is sharp.

Assume now that $\max_i\RB(u_i)=L$. Then $\QB(u_i)\le0$ for each $i$, and
\eqref{eq:transplanted-chain} gives $\QS(Tu_i)\le0$, that is,
$\RS(Tu_i)\le\mu$. Proposition~\ref{prop:saturation} yields
$\QS(Tu_i)=0$ for all three indices, so the chain
$0=\QS(Tu_i)\le c\QB(u_i)\le0$ consists entirely of equalities:
\begin{equation}\label{eq:all-form-equalities}
 \QB(u_i)=\QS(Tu_i)=0,
 \qquad c\QB(u_i)-\QS(Tu_i)=0,
 \qquad i=1,2,3.
\end{equation}

\subsection{Equality: separation of variables}

By item~\ref{item:rigidity} of Proposition~\ref{prop:full-comparison},
\begin{equation}\label{eq:separated-triple}
 u_i(r,\phi)=R(r)f_i(\phi),\qquad f_i\in H^1(\Circle).
\end{equation}
Since $R>0$ in $(0,1)$, the functions $f_i$ are nonnegative, nonzero,
and pairwise segregated.

To identify the angular factors, we use the radial energy identity
\eqref{eq:radial-integrated-identity}, proved in
Section~\ref{sec:angular-matching} while determining the constant $c$.
All its integrals are finite by \eqref{eq:R-origin}--\eqref{eq:R-boundary}.
On each annulus $a<r<b$, where $R$ is smooth, the polar derivatives of
$u_i=Rf_i$ are $\partial_ru_i=R'f_i$ and $\partial_\phi u_i=Rf_i'$.
Fubini's theorem and monotone convergence as the annuli exhaust $\Disk$ give
\begin{align*}
 \int_{\Disk}\norm{\nabla u_i}^2\dd x
 &=\int_0^1r(R')^2\dd r\int_0^{2\pi}f_i^2\dd\phi
  +I\int_0^{2\pi}\abs{f_i'}^2\dd\phi ,\\
 \int_{\Disk}u_i^2\dd x&=\int_0^1rR^2\dd r\int_0^{2\pi}f_i^2\dd\phi ,
\end{align*}
and \eqref{eq:radial-integrated-identity} gives
\[
 \QB(Rf_i)=I\left(\int_0^{2\pi}\abs{f_i'}^2\dd\phi
                   -\nu^2\int_0^{2\pi}f_i^2\dd\phi \right).
\]
Since $\QB(u_i)=0$ and $I>0$,
\begin{equation}\label{eq:angular-equalities}
 \int_0^{2\pi}\abs{f_i'}^2\dd\phi
       =\frac94\int_0^{2\pi}f_i^2\dd\phi ,
       \qquad i=1,2,3.
\end{equation}
The equality case is now a one-dimensional problem on the circle.

\subsection{Angular rigidity}

We use the Dirichlet Poincar\'e inequality on an
interval. The first Dirichlet eigenvalue of $(0,\ell)$ is
$\pi^2/\ell^2$, and its eigenspace is spanned by
$\sin(\pi t/\ell)$. Hence, by the variational characterization,
\begin{equation}\label{eq:interval-poincare}
 \int_0^\ell \abs{f'}^2\dd t
 \ge \frac{\pi^2}{\ell^2}\int_0^\ell f^2\dd t,
 \qquad f\in H^1_0(0,\ell).
\end{equation}
Equality holds precisely for scalar multiples of
$\sin(\pi t/\ell)$; for a nonzero nonnegative function, the
scalar is positive.

\begin{lemma}\label{lem:angular-rigidity}
Let $f_1,f_2,f_3\in H^1(\Circle)$ be nonnegative, nonzero, and pairwise
segregated. If \eqref{eq:angular-equalities} holds for
$i=1,2,3$, then,
after relabeling, their positivity sets are three consecutive angular
intervals $(\alpha_i,\alpha_i+2\pi/3)$, with angles interpreted modulo
$2\pi$. On each such interval, the corresponding function is a positive
multiple of $\sin\bigl(\tfrac32(\phi-\alpha_i)\bigr)$.
\end{lemma}

\begin{proof}
Every $H^1$ function on the circle has an absolutely continuous
representative, and we use these representatives. Their positivity sets
\[
 E_i=\set{\phi\colon f_i(\phi)>0}
\]
are nonempty, open, and pairwise disjoint, since an overlap would contain
an interval of positive measure. No $E_i$ is the whole circle.

Write $m_i=\measure{E_i}$ for its angular length. The set $E_i$ is the disjoint union of at most
countably many open intervals $\mathcal I_{ij}$, of lengths $\ell_{ij}$. The
restriction of $f_i$ to each such interval belongs to $H^1_0(\mathcal I_{ij})$
because it is in $H^1$ and its endpoint values vanish by continuity. Since an
absolutely continuous function has derivative zero almost everywhere on each
level set, the derivative contributes nothing on the zero set. Applying
\eqref{eq:interval-poincare} on each component gives
\begin{equation}
 \int_{\Circle}\abs{f_i'}^2\dd\phi
 \ge\sum_j\int_{\mathcal I_{ij}}\abs{f_i'}^2\dd\phi
 \ge\sum_j\frac{\pi^2}{\ell_{ij}^2}
                   \int_{\mathcal I_{ij}}f_i^2\dd\phi
 \ge\frac{\pi^2}{m_i^2}\int_{\Circle}f_i^2\dd\phi .
 \label{eq:angular-support-estimate}
\end{equation}
The last step uses $\ell_{ij}\le m_i$ and $f_i=0$ outside $E_i$.
Combining \eqref{eq:angular-support-estimate} with
\eqref{eq:angular-equalities} yields $m_i\ge2\pi/3$.
Since the three sets are disjoint,
\[
 2\pi\ge m_1+m_2+m_3\ge2\pi.
\]
Hence $m_i=2\pi/3$ for each $i$, and equality holds throughout
\eqref{eq:angular-support-estimate}.

Each $E_i$ has only one component. If it had more than one, every component
would have length strictly less than $m_i$, and on any component the integral
of $f_i^2$
is positive, so at least one term in the last inequality of
\eqref{eq:angular-support-estimate} would be strictly larger than
its lower bound.

Each $E_i$ is therefore a single interval of length $2\pi/3$, and the equality
case of \eqref{eq:interval-poincare} gives the stated sine profile. Three
disjoint intervals of that length leave no angular gap of positive length, so
in circular order their endpoints coincide successively, and they are the
consecutive intervals in~\eqref{eq:Y-sectors} for some initial angle $\alpha$.
\end{proof}

\begin{proof}[Proof of Theorem~\ref{thm:functional}]
The lower bound \eqref{eq:functional-bound} was proved in the first
subsection. If $\max_i\RB(u_i)=L$, then \eqref{eq:all-form-equalities}
holds, \eqref{eq:separated-triple} and \eqref{eq:angular-equalities}
follow, and Lemma~\ref{lem:angular-rigidity} identifies the $u_i$ with
the sector ground states \eqref{eq:sector-groundstates}. Conversely, the
sector ground states form a segregated triple with Rayleigh quotient $L$.
\end{proof}

\subsection{From ground states to cells}
\label{sec:partitions}

Let $(D_1,D_2,D_3)$ be a minimizing open three-partition of the disc. Each
$\lambda_1(D_i)$ is attained by a nonnegative ground state
$u_i\in H^1_0(D_i)$. A normalized minimizing sequence is bounded in
$H^1_0(D_i)$; Rellich's theorem gives a subsequence converging weakly in
$H^1_0(D_i)$ and strongly in $L^2(D_i)$. Lower semicontinuity shows that its
limit is a minimizer, and taking the absolute value does not change the
Rayleigh quotient. By interior elliptic regularity and the strong maximum
principle, $u_i$ is continuous and strictly positive in the connected open set
$D_i$. The three zero extensions form a segregated triple. By
Section~\ref{sec:optimality} and the minimizing property of the partition,
\[
 \max_i\RB(u_i)=\max_i\lambda_1(D_i)=L.
\]
The equality case of Theorem~\ref{thm:functional} therefore gives, after
rotation and relabeling, $u_i=M_iU_i$, where $U_i$ is the ground state of
the sector $\sector_i(\alpha)$ in~\eqref{eq:sector-groundstates}. This proves
uniqueness in the sense of ground state equivalence used in~\cite{HHTnodal}.
It remains to pass from the ground states to the cells.

Almost-everywhere equality of zero-extended ground states does not
identify their open domains pointwise. For arbitrary cells, we prove
equality of Dirichlet form domains, using the strict positivity of the
sector ground state.

\begin{lemma}\label{lem:form-domain-uniqueness}
Suppose that $D\subset\Disk$ is connected and open, and let $u$ be a
nonzero nonnegative Dirichlet ground state of $D$. If the zero extension of
$u$ agrees almost everywhere with a positive multiple $U$ of the
zero-extended ground state of a sector $\sector$ of a Y-partition, then
\[
 D\subset\sector,
 \qquad H^1_0(D)=H^1_0(\sector),
\]
where both spaces are regarded as subspaces of $H^1_0(\Disk)$ by zero
extension.
\end{lemma}

\begin{proof}
The zero-extended sector function $U$ is continuous on $\Disk$, positive
exactly on $\sector$, and smooth in $\sector$. The ground state $u$ of
$D$ is continuous and strictly positive in $D$. Since the two continuous
functions agree almost everywhere in $D$, they agree at every point of
$D$. Indeed, a nonzero difference at a point would persist on a
neighborhood of positive area. Therefore $D\subset\sector$, and
the inclusion of the smooth test-function spaces gives
\[
 H^1_0(D)\subset H^1_0(\sector).
\]

For the reverse inclusion, it suffices to show that every
$\varphi\in C_c^\infty(\sector)$ belongs to $H^1_0(D)$. Because $U$ is
smooth and strictly positive in $\sector$, the function
\[
 a(x)=
 \begin{cases}
 \varphi(x)/U(x),&x\in\sector,\\
 0,&x\in\Disk\setminus\sector,
 \end{cases}
\]
belongs to $C_c^\infty(\Disk)$. Indeed, $\varphi$ vanishes in a
neighborhood of the sector boundary, and $U$ is bounded away from zero
on the compact support of $\varphi$.

The equality of zero-extended ground states means that
$U\in H^1_0(D)$, so choose $u_n\in C_c^\infty(D)$ whose zero
extensions converge to $U$ in $H^1_0(\Disk)$. Multiplication by the fixed
smooth function $a$ is continuous on $H^1(\Disk)$: the product rule
and boundedness of $a$ and $\nabla a$ give
\[
 \norm{az}_{H^1(\Disk)}\le C_a\norm{z}_{H^1(\Disk)}.
\]
Consequently,
\[
 au_n\to aU=\varphi
 \quad\text{in }H^1(\Disk).
\]
Each $au_n$ lies in $C_c^\infty(D)$, since multiplication does not
enlarge its support. Thus $\varphi\in H^1_0(D)$. Taking the closure of
$C_c^\infty(\sector)$ in $H^1_0(\Disk)$ proves the reverse inclusion
and hence equality of the two form domains.
\end{proof}

Applying Lemma~\ref{lem:form-domain-uniqueness} to the three ground states of
a minimizing partition gives, after rotation and relabeling,
\eqref{eq:cell-form-domains}, which proves the general equality statement in
Theorem~\ref{thm:main}. As a consequence, the entire Dirichlet spectrum is
preserved, not only the first eigenvalue. Indeed, the admissible subspaces in
the min--max principle are the same, and both the gradient energy and the
$L^2$ norm are evaluated after zero extension to $\Disk$. Therefore, counting
multiplicities,
\[
 \lambda_k(D_i)=\lambda_k\bigl(\sector_i(\alpha)\bigr),
 \qquad k\ge1.
\]

Equivalently, $\sector_i(\alpha)\setminus D_i$ has zero $H^1$-capacity:
the quasi-continuous representative of $U_i\in H^1_0(D_i)$ vanishes
quasi-everywhere outside $D_i$, while $U_i$ is continuous and strictly
positive inside its sector. Conversely, deleting a relatively closed set
of zero capacity leaves $H^1_0$ unchanged; see~\cite[Remark~4.2(4) and
Theorem~4.6]{KilpelainenKinnunenMartio2000}.

\subsection{Regular strong representatives}
\label{sec:regular-representatives}
Assume now that $(D_1,D_2,D_3)$ is a regular strong minimizing partition. We show
that the form-domain identification improves to $D_i=\sector_i(\alpha)$ as an
equality of open sets. Suppose, to the contrary, that
$p\in\sector_i(\alpha)\setminus D_i$. The point $p$ cannot belong to any other
cell, since every $D_j$ is contained in its corresponding sector and the sectors
are pairwise disjoint. Since the partition is strong, $p$ lies in the interface
set.

Because $\sector_i(\alpha)$ is open, a small neighborhood of $p$ is
contained in this sector. The interface is a finite union of smooth arcs
and vertices, with no isolated punctures. If $p$ lies on an arc,
we can choose a nontrivial compact smooth subarc in this neighborhood;
if $p$ is a vertex, we can do the same on one of its incident arcs,
away from the vertex. In either case there is a smooth subarc of positive length with
\[
 \Gamma\Subset\sector_i(\alpha),\qquad \Gamma\cap D_i=\varnothing.
\]

Absorb the positive normalizing constant into the sector ground state
$U_i$. By the form-domain equality, $U_i\in H^1_0(D_i)$, and there
are $w_n\in C_c^\infty(D_i)$ converging to $U_i$ in $H^1_0(\Disk)$.
For each $n$, the compact sets $\Gamma$ and $\supp w_n$ are disjoint, so
$w_n$ vanishes on a neighborhood of $\Gamma$ and has zero trace there. Choose
a tubular neighborhood $N\Subset\sector_i(\alpha)$ of $\Gamma$, avoiding all
vertices. The local trace estimate
\[
 \norm{\Trace_{\Gamma}z}_{L^2(\Gamma)}
 \le C\norm{z}_{H^1(N)}
\]
shows that $\Trace_{\Gamma}U_i=0$.

On the other hand, $U_i$ is smooth and strictly positive on
$\Gamma\Subset\sector_i(\alpha)$, so its trace is its positive
pointwise restriction. This is a contradiction. There can therefore be
no point of $\sector_i(\alpha)\setminus D_i$, and $D_i=\sector_i(\alpha)$
for each $i$. This completes the proof of Theorem~\ref{thm:main}.

\section{General angular frequency}
\label{sec:general-nu}

The construction of Section~\ref{sec:transplantation} works for any
$\nu>0$. Set $R_\nu(r)=J_\nu(j_{\nu,1}r)$, $S_\nu(\theta)=\sin^\nu\theta$,
and
\[
 c_\nu=\frac{\int_0^\pi\sin^{2\nu-1}t\dd t}
            {\int_0^1R_\nu(t)^2\frac{\dd t}{t}},
 \qquad
 \int_0^{\Theta_\nu(r)}\sin^{2\nu-1}t\dd t
       =c_\nu\int_0^r\frac{R_\nu(t)^2}{t}\dd t.
\]
Define $T_\nu$ by replacing $R$, $S$, and $\Theta$ in~
\eqref{eq:transplant} by $R_\nu$, $S_\nu$, and $\Theta_\nu$,
respectively. The proof of Proposition~\ref{prop:stretching} applies
unchanged: the pulled-back state
$Y_{\nu,a}=\sin^\nu\vartheta_a$ satisfies
\[
 -(rY_{\nu,a}')'+\frac{\nu^2}{r}Y_{\nu,a}
 =\frac{4\nu(\nu+1)a^2}{(1+a^2r^2)^2} rY_{\nu,a},
\]
where the potential on the right-hand side is strictly decreasing. The
Wronskian argument therefore gives the unimodality of $R_\nu/Y_{\nu,a}$.
The matched integrals and the positive
tail on $(1,\infty)$ yield $c_\nu R_\nu(r)^2>\sin^{2\nu}\Theta_\nu(r)$, or
equivalently $r\Theta_\nu'(r)>\sin\Theta_\nu(r)$. The arguments in
Sections~\ref{sec:sobolev} and~\ref{sec:global} extend $T_\nu$ to a
bounded map $H^1_0(\Disk)\to H^1(\Sphere)$ satisfying
$\QS^{(\nu)}(T_\nu u)\le c_\nu\QB^{(\nu)}(u)$, where the shifted forms
have thresholds $j_{\nu,1}^2$ and $\nu(\nu+1)$.

The choice $\nu=3/2$ comes from the spherical three-partition theorem, and
the general comparison does not imply that equal sectors minimize for every
number of cells. For $\nu=2$, the four equal lunes form the nodal partition
of $\sin^2\theta \sin(2\phi)$, with eigenvalue $6$. They cannot be
minimal, because a minimal spherical $k$-partition is nodal only for
$k\le2$; see~\cite[Theorem~3.7]{HHTsphere}. Hence
$\partenergy_4(\Sphere)<6$, also stated explicitly
in~\cite[Eq.~(7.29)]{HHTsphere}. The four quarter-discs do attain
$\partenergy_4(\Disk)=j_{2,1}^2$; see~\cite[Corollary~5.6 and
Section~9]{HHTnodal}.

\section{The planar harmonic oscillator}\label{ho:sec-oscillator}

We now apply the same angular-energy matching to the harmonic oscillator
on the plane. The spherical state $S(\theta)=\sin^{3/2}\theta$ and the
constants $\nu=3/2$ and $\mu=15/4$ remain unchanged. The disc radial state
$R$ is replaced by $h(r)=r^{3/2}e^{-r^2/2}$, and the interval $(0,1)$
by $(0,\infty)$. We keep $(r,\phi)$ for planar polar coordinates and
$(\theta,\phi)$ for spherical coordinates.

\subsection{The partition problem and its functional formulation}

Recall the form $q_H$, its domain $\QH$, and the cell domains $\QH(D)$
from Section~\ref{ho:intro}. Equivalently, $\QH(D)$ consists of the zero
extensions of the functions $u\in H^1_0(D)$ with $\abs{x}u\in L^2(D)$.
This follows by cutting off at infinity and then approximating on bounded
sets, where the potential is bounded. We prove Theorem~\ref{ho:main-partition}
through the following functional statement.

For $0\ne u\in\QH$, we set $\RH(u)=q_H[u]/\norm{u}_2^2$.
We use the same convention for segregated triples as in the disc problem:
the functions are nonnegative, nonzero, and pairwise disjoint almost
everywhere.

\begin{theorem}\label{ho:main-functional}
Every segregated triple $(u_1,u_2,u_3)$ in $\QH$ satisfies
\begin{equation}\label{ho:functional-bound}
 \max_{1\le i\le3}\RH(u_i)\ge5.
\end{equation}
Equality holds if and only if, after rotation and relabeling,
\begin{equation}\label{ho:functional-equality}
 u_i(r,\phi)=
 \begin{cases}
 M_i r^{3/2}e^{-r^2/2}
      \sin \Bigl(\dfrac32(\phi-\alpha_i)\Bigr),
       &(r,\phi)\in\sector_i^\infty(\alpha),\\[2mm]
 0,&\text{otherwise},
 \end{cases}
\end{equation}
where $M_i>0$ and $\alpha_i=\alpha+2\pi(i-1)/3$.
All equalities of functions are almost everywhere.
\end{theorem}

In particular, equality in the maximum forces all three Rayleigh
quotients to equal $5$. The proof below uses spherical saturation and
the same angular rigidity lemma as the disc proof, once the exact
oscillator defect identity has been established.

\subsection{Sector eigenfunctions and compactness}\label{ho:sec-prelim}

For $0<\omega<2\pi$, let
\[
 \sector_\omega^\infty=\{(r,\phi):r>0, 0<\phi<\omega\}.
\]

\begin{lemma}\label{ho:sector-spectrum}
The Dirichlet spectrum of $H$ on $\sector_\omega^\infty$, with
multiplicities, is
\begin{equation}\label{ho:sector-eigenvalues}
 2+\frac{2n\pi}{\omega}+4m,
 \qquad n=1,2,\ldots,\quad m=0,1,\ldots.
\end{equation}
In particular,
\begin{equation}\label{ho:sector-first}
 \lambda_H(\sector_\omega^\infty)=2+\frac{2\pi}{\omega},
\end{equation}
with positive ground state
$r^{\pi/\omega}e^{-r^2/2}\sin(\pi\phi/\omega)$.
\end{lemma}

\begin{proof}
Expand in the angular sine basis and put $\eta=n\pi/\omega$.
The radial equation is
\[
 -F''-\frac1rF'+\Bigl(r^2+\frac{\eta^2}{r^2}\Bigr)F=\lambda F.
\]
Writing $F(r)=r^\eta e^{-r^2/2}P(r^2)$ reduces it to
\[
 sP''+(\eta+1-s)P'+\frac{\lambda-2\eta-2}{4}P=0.
\]
The resulting eigenfunctions are
\[
 r^\eta e^{-r^2/2}L_m^{(\eta)}(r^2)\sin(n\pi\phi/\omega),
\]
where $L_m^{(\eta)}$ is a generalized Laguerre polynomial. Laguerre
orthogonality and completeness in
$L^2((0,\infty),s^\eta e^{-s}\dd s)$, together with the angular sine
expansion, give~\eqref{ho:sector-eigenvalues};
see~\cite[Chapter~V]{Szego1975}. These functions belong to the Dirichlet
form domain: their Gaussian decay controls infinity, and $\eta>0$
ensures finite energy at the vertex.
\end{proof}

The inclusion $\QH\hookrightarrow L^2(\R^2)$ is compact. Indeed, local
Rellich compactness combines with the uniform tail estimate
\begin{equation}\label{ho:tail}
 \int_{\abs{x}>R_0}\abs{u}^2\dd x
 \le R_0^{-2}\int_{\R^2}\abs{x}^2\abs{u}^2\dd x.
\end{equation}
The same compactness argument applies to each closed subspace $\QH(D)$.
Hence
$\lambda_H(D)$ is attained for every nonempty open $D$, and on a
connected $D$ a nonnegative ground state is continuous and strictly
positive in the interior, by interior elliptic regularity and the
strong maximum principle. We therefore need no existence theorem for
optimal partitions on an unbounded domain to pass from the functional
result to the partition result.

\subsection{The radial map and strict stretching}\label{ho:sec-transfer}

Put
\begin{equation}\label{ho:profiles}
 c_H=2\sqrt\pi,\qquad h(r)=r^{3/2}e^{-r^2/2}.
\end{equation}
Direct differentiation gives
\begin{equation}\label{ho:radial-oscillator}
 -h''-\frac1r h'+\Bigl(r^2+\frac{\nu^2}{r^2}\Bigr)h=5h.
\end{equation}
Together with the spherical equation~\eqref{eq:S-equation}, this is the
analogue of the pair of radial model equations used for the disc.

Define $\Theta_H:(0,\infty)\to(0,\pi)$ by
\begin{equation}\label{ho:radial-map}
 \int_0^{\Theta_H(r)}\sin^2t\dd t
       =c_H\int_0^r s^2e^{-s^2}\dd s.
\end{equation}
Since
\[
 \int_0^\pi\sin^2t\dd t=\frac\pi2,
 \qquad
 \int_0^\infty s^2e^{-s^2}\dd s=\frac{\sqrt\pi}{4},
\]
this defines an increasing smooth diffeomorphism, with endpoint limits
$\Theta_H(0+)=0$ and $\Theta_H(\infty)=\pi$. Differentiation yields
\begin{equation}\label{ho:measure-match}
 \sin^2\Theta_H(r)\Theta_H'(r)
       =c_Hr^2e^{-r^2}=c_H\frac{h(r)^2}{r}.
\end{equation}
Near the origin,
\begin{equation}\label{ho:origin-asymptotic}
 \Theta_H(r)=c_H^{1/3}r(1+o(1)).
\end{equation}
Thus we again match the angular-energy measures. In place of the
Wronskian argument, we now use a differential comparison of the
normalized cumulative densities.

\begin{lemma}[Strict radial comparison]\label{ho:strict-comparison}
For every $r>0$,
\begin{equation}\label{ho:strict-K}
 K_H(r):=\frac{r\Theta_H'(r)}{\sin\Theta_H(r)}>1.
\end{equation}
\end{lemma}

\begin{proof}
Introduce the normalized cumulative functions
\[
 F(r)=\frac4{\sqrt\pi}\int_0^r s^2e^{-s^2}\dd s,
 \qquad
 G(\theta)=\frac2\pi\int_0^\theta\sin^2t\dd t.
\]
Then $F(r)=G(\Theta_H(r))$. For $0<p<1$, put
$r=F^{-1}(p)$, $\theta=G^{-1}(p)$, and define
\begin{equation}\label{ho:IJ}
 I_H(p)=rF'(r)=\frac4{\sqrt\pi}r^3e^{-r^2},
 \qquad
 J_H(p)=\sin\theta G'(\theta)=\frac2\pi\sin^3\theta.
\end{equation}
Both functions are positive on $(0,1)$, extend continuously to $[0,1]$,
and vanish at the endpoints. Differentiation with respect to $p$ gives
\begin{equation}\label{ho:IJ-derivatives}
 I_H'=3-2r^2,\qquad I_H''=-\frac{4r^2}{I_H},\qquad
 J_H'=3\cos\theta,\qquad J_H''=-\frac{3\sin^2\theta}{J_H}.
\end{equation}
Suppose that $I_H-J_H$ has a negative minimum at an interior point.
At this point $I_H'=J_H'$, so $2r^2=3(1-\cos\theta)$. Since $I_H\le J_H$,
\begin{align*}
 (I_H-J_H)''
 =-\frac{6(1-\cos\theta)}{I_H}
       +\frac{3(1-\cos^2\theta)}{J_H}
 \le-\frac{3(1-\cos\theta)^2}{J_H}<0,
\end{align*}
contrary to the second-derivative test. Hence $I_H\ge J_H$.
An interior point of equality would again be a minimum and give the same
contradiction. Thus $I_H>J_H$ on $(0,1)$. Differentiating
$F(r)=G(\Theta_H(r))$ now gives
\[
 \frac{I_H(p)}{J_H(p)}
 =\frac{r\Theta_H'(r)}{\sin\Theta_H(r)}=K_H(r).\qedhere
\]
\end{proof}

\subsection{The defect identity on the full form domain}

In the coordinate convention already used for the disc, let
$\Psi_H(r,\phi)=(\Theta_H(r),\phi)$. This is a diffeomorphism from
$\R^2\setminus\{0\}$ to the sphere with its poles removed. For
$u\in\mathcal C_H:=C_c^\infty(\R^2\setminus\{0\})$, write $u=hw$ and set
\begin{equation}\label{ho:transplant}
 (\TransH u)(\Theta_H(r),\phi)
      =\frac{S(\Theta_H(r))}{h(r)}u(r,\phi).
\end{equation}
Define the shifted oscillator form by
\begin{equation}\label{ho:oscillator-shift}
 \QHo(u)=q_H[u]-5\norm{u}_2^2.
\end{equation}
Lemma~\ref{lem:transform}, applied to~\eqref{ho:radial-oscillator} with
$a=r$ and $V=r^3+\nu^2/r-5r$, gives
\begin{equation}\label{ho:oscillator-gsr}
 \QHo(u)=\int_0^{2\pi} \int_0^\infty
 h^2\Bigl[r\abs{w_r}^2+
       \frac1r\bigl(\abs{w_\phi}^2-\nu^2\abs{w}^2\bigr)\Bigr]
       \dd r\dd\phi.
\end{equation}
The spherical calculation in~\eqref{eq:sphere-factorization}, now using
$\Theta_H$ and~\eqref{ho:measure-match}, yields
\begin{equation}\label{ho:sphere-gsr}
 \QS(\TransH u)=c_H\int_0^{2\pi} \int_0^\infty
 h^2\Bigl[\frac{r}{K_H(r)^2}\abs{w_r}^2+
       \frac1r\bigl(\abs{w_\phi}^2-\nu^2\abs{w}^2\bigr)\Bigr]
       \dd r\dd\phi.
\end{equation}
Indeed, the radial coefficient is
$\sin^4\Theta_H/\Theta_H'=c_Hrh^2/K_H^2$, and the angular coefficient
is $\sin^2\Theta_H\Theta_H'=c_Hh^2/r$. Hence
\begin{equation}\label{ho:core-defect}
 \QS(\TransH u)=c_H\bigl(\QHo(u)-\DefH(u)\bigr),
\end{equation}
where
\begin{equation}\label{ho:defect}
 \DefH(u)=\int_0^{2\pi} \int_0^\infty
 rh(r)^2\bigl(1-K_H(r)^{-2}\bigr)
       \abs{\partial_r (u/h)}^2
       \dd r\dd\phi\ge0.
\end{equation}
The separate terms in~\eqref{ho:oscillator-gsr} need not be integrable
for an arbitrary $u\in\QH$ that is nonzero near the origin. We therefore
extend the combined defect identity, rather than its separate terms.

\begin{proposition}\label{ho:full-transfer}
Formula~\eqref{ho:transplant} extends to a bounded, injective linear map
\[
 \TransH:\QH\to H^1(\Sphere).
\]
The extension is given by the same formula almost everywhere and
preserves nonnegativity and segregation. For every $u\in\QH$, the
integral~\eqref{ho:defect} is finite and
\begin{equation}\label{ho:full-defect}
 \QS(\TransH u)=c_H\bigl(\QHo(u)-\DefH(u)\bigr)
       \le c_H\QHo(u).
\end{equation}
Moreover, $\DefH(u)=0$ if and only if $\partial_r(u/h)=0$ almost
everywhere on every compact annulus centered at the origin.
\end{proposition}

\begin{proof}
Cutting off at infinity and mollifying shows that $C_c^\infty(\R^2)$
is dense in $\QH$. Lemma~\ref{lem:punctured-core}, with $V(x)=\abs{x}^2$,
then gives density of $\mathcal C_H$ in $\QH$.

A change of variables using~\eqref{ho:measure-match} gives
\begin{equation}\label{ho:transplant-L2}
 \norm{\TransH u}_{L^2(\Sphere)}^2
       =c_H\int_{\R^2}\frac{\sin^2\Theta_H(r)}{r^2}\abs{u(x)}^2\dd x,
       \qquad u\in\mathcal C_H.
\end{equation}
The weight is bounded near $0$ by~\eqref{ho:origin-asymptotic}, and is
at most $1$ for $r\ge1$. Put
$b_H=\sup_{r>0}\sin^2\Theta_H(r)/r^2<\infty$.
By~\eqref{ho:core-defect},
\[
 \norm{\nabla_{\Sphere}\TransH u}_2^2
 =\QS(\TransH u)+\mu\norm{\TransH u}_2^2
 \le c_Hq_H[u]+\mu c_Hb_H\norm{u}_2^2.
\]
Together with~\eqref{ho:transplant-L2}, this proves boundedness on the
core and hence the existence of a bounded extension. On compact annuli
the coordinate map and multiplier are smooth and strictly positive.
Local convergence identifies the extension with~\eqref{ho:transplant}
almost everywhere. This proves injectivity and preservation of
nonnegativity and segregation, and~\eqref{ho:transplant-L2} passes to
the limit.

For the exact identity, take $u_n\in\mathcal C_H$ with $u_n\to u$ in
$\QH$. Applying~\eqref{ho:core-defect} to $u_n-u_m$ gives
\[
 \DefH(u_n-u_m)=\QHo(u_n-u_m)
       -c_H^{-1}\QS\bigl(\TransH(u_n-u_m)\bigr)\to 0.
\]
Thus $\partial_r(u_n/h)$ is Cauchy in the weighted $L^2$ space whose
squared norm is~\eqref{ho:defect}. On compact annuli it converges to
$\partial_r(u/h)$ by ordinary $H^1$ convergence, so the weighted limit
is this derivative. Therefore $\DefH(u_n)\to\DefH(u)$, and passing to
the limit in~\eqref{ho:core-defect} proves~\eqref{ho:full-defect}.
The last assertion follows from strict positivity of the weight for
all $r>0$.
\end{proof}

\begin{lemma}[Vanishing radial derivative]\label{ho:angular-dependence}
If $u\in\QH$ and $\DefH(u)=0$, there is an $f\in H^1(\Circle)$ such that
\begin{equation}\label{ho:angular-form}
 u(r,\phi)=h(r)f(\phi)\qquad\text{a.e. in }\R^2.
\end{equation}
For this function,
\begin{equation}\label{ho:angular-energy}
 \QHo(hf)=\frac{\sqrt\pi}{4}
       \int_0^{2\pi}\Bigl(\abs{f'}^2-\frac94\abs{f}^2\Bigr)\dd\phi.
\end{equation}
\end{lemma}

\begin{proof}
The annular averaging argument in the proof of
item~\ref{item:rigidity} of Proposition~\ref{prop:full-comparison} applies
unchanged to $u/h$ on an exhaustion of the plane by compact annuli.
It gives a single angular factor $f\in H^1(\Circle)$.
For $u=hf$, all radial integrals in the ground-state calculation are
finite, and the endpoint terms vanish. Equation~\eqref{ho:radial-oscillator}
and $\int_0^\infty h(r)^2\dd r/r=\sqrt\pi/4$ then give
\eqref{ho:angular-energy}. One can first take $f$ smooth and then use
density in $H^1(\Circle)$.
\end{proof}

\subsection{Minimality and rigidity}\label{ho:sec-proof}

\begin{proof}[Proof of Theorem~\ref{ho:main-functional}]
Let $(u_1,u_2,u_3)$ be a nonnegative, nonzero, segregated triple in
$\QH$. Proposition~\ref{ho:full-transfer} gives nonnegative, nonzero,
segregated functions $g_i=\TransH u_i\in H^1(\Sphere)$ and the identities
\begin{equation}\label{ho:triple-defect}
 \QS(g_i)=c_H\bigl(\QHo(u_i)-\DefH(u_i)\bigr)
 \le c_H\QHo(u_i),\qquad i=1,2,3.
\end{equation}
If $\RH(u_i)<5$ for every $i$, then all three $\QS(g_i)$ are strictly
negative, contradicting Lemma~\ref{lem:spherical-functional}. This proves
\eqref{ho:functional-bound}. Lemma~\ref{ho:sector-spectrum}, with
$\omega=2\pi/3$, shows that the triple~\eqref{ho:functional-equality}
attains the bound.

Suppose now that $\max_i\RH(u_i)=5$. Then $\QHo(u_i)\le0$ for every
$i$, so~\eqref{ho:triple-defect} gives $\QS(g_i)\le0$, equivalently
$\RS(g_i)\le\mu$. Spherical saturation
(Proposition~\ref{prop:saturation}) therefore yields $\QS(g_i)=0$ for
all three components. Since $\DefH(u_i)\ge0$ and $\QHo(u_i)\le0$,
the exact identities~\eqref{ho:triple-defect} force
\begin{equation}\label{ho:all-form-equalities}
 \QHo(u_i)=\DefH(u_i)=0,\qquad i=1,2,3.
\end{equation}
By Lemma~\ref{ho:angular-dependence}, there are
$f_i\in H^1(\Circle)$ such that $u_i(r,\phi)=h(r)f_i(\phi)$ almost
everywhere. Because $h(r)>0$ for $r>0$, the $f_i$ are nonnegative,
nonzero, and segregated. Equations~\eqref{ho:angular-energy}
and~\eqref{ho:all-form-equalities} give
\[
 \int_{\Circle}\abs{f_i'}^2\dd\phi
   =\frac94\int_{\Circle}\abs{f_i}^2\dd\phi,\qquad i=1,2,3.
\]
These are exactly the hypotheses of Lemma~\ref{lem:angular-rigidity}.
After a common rotation and relabeling, the positivity sets of the
$f_i$ are the three consecutive intervals of length $2\pi/3$, and each
$f_i$ is a positive multiple of the first sine function on its interval.
Multiplying by $h$ gives~\eqref{ho:functional-equality}. Conversely,
these sector ground states all have Rayleigh quotient $5$, as already
noted, completing the equality characterization.
\end{proof}

\subsection{From ground states to cells}\label{ho:sec-partitions}

Let $\mathcal P=(D_1,D_2,D_3)$ be an open partition. By compactness,
choose nonnegative ground states $u_i\in\QH(D_i)$.
Theorem~\ref{ho:main-functional} gives $\Lambda_H(\mathcal P)\ge5$,
and the Y-partition attains $5$.

If $\Lambda_H(\mathcal P)=5$, the functional equality statement
identifies the zero-extended ground states with
\eqref{ho:functional-equality}. Absorb their positive constants into
sector states $U_i$. The function $U_i$ is continuous on the plane and
positive exactly on $\sector_i^\infty(\alpha)$. Interior continuity and
positivity of $u_i$ show, just as in
Lemma~\ref{lem:form-domain-uniqueness}, that
$D_i\subset\sector_i^\infty(\alpha)$.

The same multiplication argument proves equality of the form domains.
For $\varphi\in C_c^\infty(\sector_i^\infty(\alpha))$, the quotient
$a=\varphi/U_i$, extended by zero outside the sector, lies in
$C_c^\infty(\R^2)$. Multiplication by $a$ is continuous on $\QH$:
the $H^1$ estimate is as before, and
\[
 \norm{\abs{x}az}_2\le\norm{a}_\infty\norm{\abs{x}z}_2.
\]
Approximating $U_i\in\QH(D_i)$ by functions in $C_c^\infty(D_i)$ and
multiplying by $a$ gives $\varphi\in\QH(D_i)$. Closure and the opposite
inclusion from $D_i\subset\sector_i^\infty(\alpha)$ prove
\eqref{ho:cell-form-domains}.

For a regular strong representative, the local trace argument in
Section~\ref{sec:regular-representatives} applies without change:
any additional interface point inside a sector would give a smooth
subarc there on which $U_i$ has both zero trace and a strictly positive
pointwise value. All arguments take place in a bounded neighborhood,
so the unbounded cells cause no difficulty. Thus the cells are the
sectors as open sets. This proves Theorem~\ref{ho:main-partition}.

As in the capacitary interpretation following
Lemma~\ref{lem:form-domain-uniqueness},
$\sector_i^\infty(\alpha)\setminus D_i$ has zero $H^1$-capacity.
Conversely, removing a relatively closed set of zero capacity preserves
$\QH(\sector_i^\infty(\alpha))$, by cutting off at infinity and using
boundedness of the potential on compact sets.

\subsection{Scaling, translation, and the Gaussian formulation}
\label{ho:sec-consequences}

\begin{corollary}\label{ho:scaling}
Let $a,b>0$, $c\in\R$ and $x_0\in\R^2$. For
\[
 H_{a,b,c,x_0}=-a\Delta+b\abs{x-x_0}^2+c,
\]
the minimal three-partition energy is $5\sqrt{ab}+c$.
Every regular strong minimizing representative is, up to rotation and
relabeling, the translated Y-partition with vertex $x_0$.
\end{corollary}

\begin{proof}
The change of variables $x=x_0+(a/b)^{1/4}y$, with the corresponding
unitary normalization in $L^2$, transforms the operator into
$\sqrt{ab}(-\Delta_y+\abs{y}^2)+c$. Apply
Theorem~\ref{ho:main-partition}.
\end{proof}
In particular, the value for $-\Delta+\omega^2\abs{x}^2$, with $\omega>0$,
is $5\omega$.

For the Gaussian formulation, let
$\dd\gamma=\pi^{-1}e^{-\abs{x}^2}\dd x$, and let $H^1(\gamma)$ be the
completion of $C_c^\infty(\R^2)$ for the squared norm
$\int(\abs{\nabla f}^2+\abs{f}^2)\dd\gamma$. The map
\[
 (Uf)(x)=\pi^{-1/2}e^{-\abs{x}^2/2}f(x)
\]
is unitary from $L^2(\gamma)$ to $L^2(\R^2)$ and preserves supports.
Integration by parts on the core, followed by completion, gives
\begin{equation}\label{ho:gaussian-transform}
 q_H[Uf]=\int_{\R^2}\abs{\nabla f}^2\dd\gamma
              +2\int_{\R^2}\abs{f}^2\dd\gamma.
\end{equation}
The Gaussian Dirichlet form on a cell $D$ is defined by closing
$C_c^\infty(D)$ in $H^1(\gamma)$. Multiplication by the positive smooth
factor in $U$ maps these cores onto the corresponding oscillator cores;
\eqref{ho:gaussian-transform} therefore identifies their form domains.
It follows that the minimal three-partition energy for
$\int\abs{\nabla f}^2\dd\gamma$, equivalently for $-\Delta+2x\cdot\nabla$,
is $3$, with the same minimizing sectors and the same form-domain
uniqueness convention. For the standard Gaussian density
$(2\pi)^{-1}e^{-\abs{x}^2/2}$ and the operator $-\Delta+x\cdot\nabla$, the
value is $3/2$, by dilation. These are spectral partition statements;
no Gaussian perimeter functional is involved.

\subsection{The common radial-profile principle}\label{ho:sec-general-profile}

Both applications are instances of the same conditional comparison.
Let $0<b\le\infty$ and $\Omega_b=\{x\in\R^2:\abs{x}<b\}$, with
$\Omega_\infty=\R^2$. Consider a closed, semibounded radial Dirichlet Schr\"odinger form
\[
 q_V[u]=\int_{\Omega_b}\bigl(\abs{\nabla u}^2+V(\abs{x})\abs{u}^2\bigr)\dd x.
\]
Suppose it has a positive radial factor $\rho$ satisfying
\begin{equation}\label{ho:general-profile}
 -\rho''-\frac1r\rho'+\Bigl(V(r)+\frac{9}{4r^2}\Bigr)\rho
       =\lambda_*\rho,\qquad
 0<\int_0^b\frac{\rho(r)^2}{r}\dd r<\infty.
\end{equation}
Define $\Theta_\rho$ and $c_\rho$ by
\begin{equation}\label{ho:general-map}
 \int_0^{\Theta_\rho(r)}\sin^2t\dd t
   =c_\rho\int_0^r\frac{\rho(s)^2}{s}\dd s,
 \qquad
 c_\rho=\frac{\pi/2}{\int_0^b\rho(s)^2\dd s/s},
\end{equation}
and put
\[
 (T_\rho u)(\Theta_\rho(r),\phi)
   =\frac{S(\Theta_\rho(r))}{\rho(r)}u(r,\phi),
 \qquad
 K_\rho(r)=\frac{r\Theta_\rho'(r)}{\sin\Theta_\rho(r)}.
\]
For smooth functions compactly supported away from the radial endpoints,
Lemma~\ref{lem:transform} gives
\begin{equation}\label{ho:general-identity}
 \QS(T_\rho u)=c_\rho\Big(q_V[u]-\lambda_*\norm{u}_2^2
 -\int_0^{2\pi} \int_0^b
 r\rho^2(1-K_\rho^{-2})\abs{\partial_r(u/\rho)}^2\dd r\dd\phi\Big).
\end{equation}
Assume that the punctured smooth core is dense in the source form
domain, that $K_\rho\ge1$, and that this transplantation extends
continuously from that form domain into $H^1(\Sphere)$, with the same
positive-multiplier formula almost everywhere. Then
Lemma~\ref{lem:spherical-functional}, applied to the transplanted
functions, gives the lower bound $\lambda_*$ for every three-partition.
If the equal-sector ground states belong to
the corresponding cell form domains and have eigenvalue $\lambda_*$,
the bound is sharp.

For the disc, $(b,V,\rho,\lambda_*)=(1,0,R,L)$; for the oscillator,
they are $(\infty,r^2,h,5)$. The endpoint, density, and strict
radial-comparison hypotheses have been checked above in both cases.
They must be checked anew for any other radial potential.

\section*{Acknowledgments}

I acknowledge the use of AI tools for mathematical exploration and assistance
with preparing this manuscript. I thank Bernard Helffer for reading an early
version and for suggesting several improvements, including the suggestion
to try the same method also for the harmonic oscillator.

\bibliographystyle{amsplain}
\bibliography{S-radial}

\end{document}